\documentclass{amsart}

\usepackage{dsfont}

\def\N{\mathbb N}
\def\R{\mathbb R}

\def\RR{\mathcal{R}}
\def\CC{\mathcal{C}}

\def\mc#1{{\mathcal #1}}

\newcommand{\yy}{{\bf y}}
\newcommand{\zz}{{\bf z}}
\newcommand{\ww}{{\bf w}}

\newcommand{\wt}[1]{{\widetilde{#1}}}

\newcommand{\eps}{\varepsilon}
\newcommand{\g}{\gamma}
\newcommand{\Ga}{\Gamma}

\theoremstyle{plain}
\newtheorem{theorem}{Theorem}[section]
\newtheorem{lemma}[theorem]{Lemma}
\newtheorem{definition}[theorem]{Definition}
\newtheorem{corollary}[theorem]{Corollary}
\newtheorem{proposition}[theorem]{Proposition}
\theoremstyle{definition}
\newtheorem{remark}[theorem]{Remark}

\theoremstyle{plain}

\newtheorem{claim}[theorem]{Claim}
\newtheorem{example}[theorem]{Example}

\numberwithin{equation}{section}

\begin{document}

\title[Solutions of ODE with regular separation: interlacement vs Hardy]{Solutions of definable ODEs with regular separation and dichotomy interlacement versus Hardy}

\author{Olivier Le Gal}
\address{Laboratoire de Math{\'e}matiques, Universit{\'e} de Savoie,	B{\^a}timent Chablais, Campus Scientifique, 73376 Le Bourget-du-Lac	Cedex, France }\email{olivier.le-gal@univ-savoie.fr}
  \author{Micka\"el Matusinski}
 \address{Univ. Bordeaux, CNRS, Bordeaux INP, IMB, UMR 5251, F-33400 Talence, France} \email{mickael.matusinski@math.u-bordeaux.fr}
 \author{Fernando Sanz S\'anchez}
\address{Dept. \'{A}lgebra, An\'{a}lisis Matem\'{a}tico, Geometr\'{\i}a y Topolog\'{\i}a, Instituto de Matemáticas de la UVa (IMUVA), Universidad de Valladolid, Spain}\email{fsanz@agt.uva.es}




\begin{abstract}
We introduce a notion of regular separation for solutions of systems of ODEs $y'=F(x,y)$, where $F$ is definable in a polynomially bounded o-minimal structure and $y=(y_1,y_2)$. Given a pair of solutions with flat contact, we prove that, if one of them has the property of regular separation, the pair is either interlaced or generates a Hardy field. We adapt this result to trajectories of three-dimensional vector fields with definable coefficients. In the particular case of real analytic vector fields, it improves the dichotomy interlaced/separated of certain integral pencils {obtained by F. Cano, R. Moussu and the third author}. In this
context, we show that the set of trajectories with the regular separation property and asymptotic to a formal invariant curve is never empty and it is represented by a subanalytic set of minimal dimension containing the curve. Finally, we show how to construct examples of formal invariant curves which are transcendental with respect to subanalytic sets, using the so-called (SAT) property {introduced by J.-P. Rolin, R. Shaefke and the third author}.
\end{abstract}

\maketitle

\section{Introduction.}\label{sec:preliminaries}

We consider a system of two ordinary differential equations
\begin{equation}\label{SF}\tag{${S}_F$}
\bigg\{	\begin{aligned}
	{y_1}'&=&f_1(x,y_1,y_2)\\
	{y_2}'&=&f_2(x,y_1,y_2)
	\end{aligned} 
\end{equation} 
where $F=(f_1,f_2) : \Omega\to\mathbb R^2$ is a $C^1$ map
on some open set $\Omega\subset \mathbb R_+\times\mathbb R^2$ with $(0,0,0)\in\overline{\Omega}$. A \emph{solution at $0$ of} \eqref{SF} (sometimes called simply a  \emph{solution}) is a $C^1$ map $\gamma:(0,a)\to\mathbb R^2$, whose graph is contained in $\Omega$, and satisfying the system \eqref{SF}. We are interested in this article in the relative behavior of two distinct solutions of \eqref{SF}, following similar studies addressed in various contexts, in particular by M. Rosenlicht \cite{Ros} and M. Boshernitzan \cite{Bos} in the setting of Hardy fields, and by F. Cano, R. Moussu and F. Sanz for analytic vector fields \cite{Can-Mou-San1, Can-Mou-San2}.

The possible relative behaviors of the solutions strongly depend on the nature of the function $F$. When \eqref{SF} is an autonomous linear system, i.e. $F$ is a linear map in $(y_1,y_2)$ and does not depend on the variable $x$, C. Miller \cite{Mil:traj} and M. Tychonievic \cite{Tyc} give a complete classification of the solutions in model theoric terms, which refines the classical dichotomy ``spiraling'' versus ``non-oscillating'' for trajectories of planar vector fields.
If $F$ is a linear map in $(y_1,y_2)$ whose coefficients are functions of $x$ definable in an o-minimal structure, O. Le Gal, F. Sanz and P.  Speissegger \cite{Leg-San-Spe} show that either any {two given} solutions of \eqref{SF} {are interlaced}, or all solutions belong to a common o-minimal structure. 
This confirms the following heuristic statement \eqref{equH} (for ``tame'' differential systems) that can already be pulled out from \cite{Can-Mou-San1,Can-Mou-San2}: 
	\begin{equation}\label{equH}
		\tag{H}\textsl{Relative oscillation of two solutions requires infinite twisting.}
	\end{equation}
Here we focus on systems \eqref{SF} that we call \emph{definable systems}, where $F$ is a tame function in the broader sense of o--minimal geometry  \cite{vdd-mil, Cos}.  More precisely, $F$ is assumed to be a function definable in a given \emph{o-minimal and polynomially bounded expansion $\mathcal R$ of the field of real numbers} (see {the following} section for the specific properties that we will use here). {From now on, definability {will} always refer to definability with parameters in the structure $\mathcal R$}.  Let us mention that this comprises the aforementioned contexts of linear, semialgebraic or globally subanalytic functions. We introduce two notions in this context: regular separation -- inspired by a similar property developed by {S. \L{}ojasiewicz} in semi-analytic geometry, see for instance \cite{Loj} --, and flat contact.

\begin{definition}\label{def:regular-flat}~
	
	\begin{enumerate}
		\item {Let $\gamma $ be a  solution of  a definable system \eqref{SF}. We say that $\gamma$ has} \emph{regular separation property} (with respect to $\mathcal R$) if for any {map} $f:\R^3\to\R$ definable in $\RR$ whose domain contains the graph of $\gamma$, there exist{s} a {real number} $a>0$ such that {either 
			\[  \forall x\in (0,a),\, f(x,\gamma(x))=0 \]
			or  \[   \exists k\in\mathbb N, \forall x\in (0,a),\; |f(x,\gamma(x))|\geq x^k.\]}
		\item {Let $\gamma,\delta$ be two different  solutions of  a definable system \eqref{SF}, and let $\eps:=\delta-\g$.} $\gamma$ and $\delta$ are said to have {\em flat contact} if:\[ \forall k\in\mathbb N,\, \exists a >0,\, \forall x\in(0,a),\; \|\eps(x)\|\leq x^k.\]
	\end{enumerate}
\end{definition}
\noindent 

In order to give a precise statement, let us introduce, for a given pair of solutions, the notion of interlacement and the associated ring of definable germs.

\begin{definition}\label{def:interlaced-separated}
	Let $\gamma,\delta$ be two different  solutions of \eqref{SF}, and $\eps:=\delta-\g$.
	\begin{enumerate}
		\item We say that $\gamma$ and $\delta$ are \emph{interlaced} if the plane curve  $x\mapsto \eps(x)$, for $x\in(0,a)$ and $a>0$, spirals infinitely around the origin. In other terms, any continuous measure $\theta:(0,a)\to\mathbb R$ of the angle between $(1,0)$ and $\Theta(x):=\frac{\varepsilon(x)}{||\varepsilon(x)||}$ satisfies $\lim_{x\to 0} \theta(x)=+\infty\text{ or }-\infty$.
		\item We call \emph{ring of definable germs over $\gamma,\delta$} the ring $\mathcal F(x,\gamma,\delta)$ consisting of the germs at $0^+$ of  compositions $x\mapsto f(x,\gamma(x),\delta(x))$ with $f$ ranging over all definable functions from $\mathbb R^5$ to $\mathbb R$ whose domain contains the graph of $x\mapsto (\gamma(x),\delta(x))$.
		
	\end{enumerate}
\end{definition}

Let us recall that a Hardy field is a subring of the ring of germs of functions $h:\R\to\R$ at $0^+$, which is a field and is closed under the natural derivation \cite{Ros}.
The main result of the present paper is the following one. It provides another instance of the heuristic principle \eqref{equH} in the context of definable systems of differential equations.

\begin{theorem}\label{th:main}
	Let \eqref{SF} be a differential system {as above} where $F$ is definable in a polynomially bounded o-minimal structure $\mathcal R$ expanding $\R$, and let $\gamma$ be a solution at 0 of \eqref{SF} {which} has the regular separation property. Let $\delta$ be another solution of \eqref{SF} which has flat contact with $\gamma$. Then, either $\mathcal{F}(x,\gamma,\delta)$ is a Hardy field, or else $\gamma$ and $\delta$ are interlaced.
\end{theorem}

We prove Theorem \ref{th:main} in Section \ref{sec:proof-main}.

In Section \ref{sec:trajectories}, we adapt Theorem \ref{th:main} to trajectories of three--dimensional definable vector fields. {In this way, we partially generalize the results about} non-oscillating trajectories of analytic vector fields by F. Cano, R. Moussu and the third author \cite{Can-Mou-San1,Can-Mou-San2}. The main result is Theorem \ref{th:defvf}{. Under the hypothesis that one of the trajectories has regular separation property, it provides} a more precise version of the dichotomy ``enlac\'e'' (our {present} notion of interlacement) versus ``s\'epar\'e'' established in \cite{Can-Mou-San2} for a pair of trajectories having flat contact.

In Section \ref{sec:integral-pencils}, we focus on the family $\mathcal{P}_{\mathcal{C}}$ of trajectories of an analytic vector field which are asymptotic to a given formal curve $\mathcal{C}$. Our objective is to describe the subset $\mathcal T$ of trajectories in $\mathcal{P}_{\mathcal{C}}$ which have regular separation {property}. We {realize} {$\mathcal T$ by determining} a subanalytic set $S$ which contains a representative of a trajectory if and only if it belongs to $\mathcal T$, and prove  that $\mathcal T$ is never empty (Theorem \ref{pro:formal-axis}).
The dimension of $S$ measures the transcendence of $\mathcal{C}$ with respect to subanalytic sets, and we give some implications of our result in the two cases where dim$(S)=1$ (analytic axis), and dim$(S)=3$ (transcendent axis).

{Accordingly, analytic vector fields having a transcendent formal invariant curve  provide integral pencils of maximal dimension for which all members satisfy the hypothesis of Theorem \ref{th:main}. The purpose of the last Section \ref{sec:examples} is to exhibit examples of this kind.} 
{Even if} this case is probably generic, it is not {easy} to prove that a particular formal curve is  transcendent with respect to subanalytic sets. We derive {this property} from a strong analytic transcendence condition (SAT) introduced and studied in \cite{Rol-San-Sch}. 


\section{Proof of Theorem~\ref{th:main}.}\label{sec:proof-main}


\subsection{Preliminary properties.}

Since we  deal with properties of germs, we often need to restrict to smaller domains. In {order to ease the reading}, we interchangeably use the expressions ``ultimately'',  ``for small $x$'', ``for $x$ close to $0$'', that must be understood as ``there exists a positive {real} $a$ such that, for $x\in(0,a)$''. {Likewise}, a curve $c:(0,a)\to M$ ``meets infinitely many times'' a subset $X\subset M$ if there is an infinite sequence $(t_n) \in (0,a)^{\mathbb N}$ {tending to} $0$ such that $c(t_n)\in X$.

We assume the reader to be familiar with the basic notions and properties of o-minimal geometry as presented e.g. in \cite{vdd-mil,Cos}. We consider a polynomially bounded o-minimal structure $\mathcal R$ expanding the field of real number $\R$. Recall that a function is called definable if its graph is a definable set. 
For instance, the {characteristic} function of a definable set $X$, this is, the map $\mathds{1}_X$
defined by $\mathds{1}_X(x)=1$ if $x\in X$, $\mathds{1}_X(x)=0$ if $x\notin X$, is a definable function. 
The sign function $\mathbb R\ni x\mapsto \mathrm{sign}(x)\in\{+,-,0\}$ is {also} definable, up to identifying $\{+,-,0\}$ with $\{1,-1,0\}$ in the obvious way.
We will {repeatedly} use the fact that various operations {like} taking the supremum, passing to the limit or differentiating, produce definable functions as soon as their entries are definable. 
The main tool of o-minimal geometry we use is the existence of a \emph{cell decomposition} for definable sets.
We refer the reader to \cite{vdD} for details, and to \cite[Theorem 6.7]{Cos} for the particular version of the {Cell Decomposition Theorem that} we use here. 
We denote by $||\cdot||$ the Euclidean norm on $\mathbb R^n$.\\

The following lemma shows that regular separation implies {non-oscillation} with respect to definable sets:  ultimately, a curve that has regular separation {property} is either included in a definable set $X$ or else does not meet $X$.
\begin{lemma}\label{lem:1}
	Let $\gamma$ be a solution of a definable system \eqref{SF} which has the regular separation property, and $X$ a definable subset of $\mathbb R^3$. Suppose that there exists an infinite sequence $(x_n)\to 0$ such that $(x_n,\gamma(x_n))\in X$. Then there exists $a>0$, such that $X$ contains the graph of $\gamma$ for $x<a$.
\end{lemma}
\begin{proof}
	Set $g(x):=1-\mathds{1}_X(x,\gamma(x))$, where $\mathds{1}_X$ is the {characteristic} function of $X$. Then $g$ is the composition of a definable function with $\gamma$. It is either bounded from below by a power of $x$ or identically $0$. Since it vanishes on $(x_n)$, it cannot be bounded from below. So $g=0$ for small $x$, this means, $X$ contains the graph of $\gamma$ for small $x$.
\end{proof}

\begin{lemma}\label{lem:2}
	Let $\gamma$ be a solution of a definable system \eqref{SF} which has regular separation property, and $S:\mathbb R^4\to\mathbb R$ be a definable map.
	Then there exists $a>0$, $m>0$ and a sign $\alpha\in\{-,0,+\}$ such that,
	\[ \forall x\in (0,a),\; \forall z\in\mathbb R,\; 0<z<x^m\Rightarrow \mathrm{sign}(S(x,\gamma(x),z))=\alpha.\]
\end{lemma}
\begin{proof}
	Let $s(x,y_1,y_2)=\lim_{z\to 0, z>0} (\mathds 1_{\mathbb R^+}-\mathds 1_{\mathbb R^-}) \circ S(x,y_1,y_2,z)$. Then $s$ is definable, well defined
	and has a value in $\{-1,0,1\}$ depending on the sign of $S(x,y_1,y_2,z)$ for small positive $z$. {By the}  regular separation property, $s(x,\gamma(x))$ has ultimately constant value, then constant sign $\alpha$ for small $x$.
	
	Now, let $r(x,y_1,y_2)=\sup\{z\le 1;\;\forall z',\; 0<z'<z\Rightarrow \mathrm{sign}(S(x,y_1,y_2,z))=\alpha\}$. Again, $r$ is definable, and
	moreover, for small $x$, $r(x,\gamma(x))>0$ by definition of $\alpha$. So, from regular separation, there exists $a>0$, $m>0$, such that $x\in(0,a)\Rightarrow r(x,\gamma(x))>x^m$, which is the {statement} of the lemma.
\end{proof}

We will also use the following elementary result ``\`{a} la Rolle'', {whose proof is left to the reader}:
\begin{proposition}\label{prop:0}
	Let $f:\mathbb R\to\mathbb R$ be a differentiable map and $a< b$ be two consecutive zeros of $f$. Then $f'(a)f'(b)\le 0$.
\end{proposition}

\subsection{The proof.}

We first reduce the proof to that of the following seemingly weaker property.

\begin{proposition}\label{prop:sign}
	With the same hypothesis as {in} Theorem \ref{th:main}, either any germ in $\mathcal{F}(x,\gamma,\delta)$ has an ultimate sign, or else $\g$ and $\delta$ are interlaced.
\end{proposition}
\begin{proof}[Proof of (Proposition \ref{prop:sign} $\; \Rightarrow$ Theorem \ref{th:main})]
	
	Suppose that $\g,\delta$
	are not interlaced. Since any germ $\varphi$ in $\mathcal{F}(x,\g,\delta)$ has an ultimate sign, one can define the germ of $1/\varphi$ as soon as the germ of $\varphi$ is not zero, i.e., as soon as $\varphi(x)\neq 0$ for any $x$ close enough to 0. So $\mathcal{F}(x,\gamma,\delta)$ is a field. It remains to show that this field is closed under derivation.
	
	For this, let $g:\mathbb R^5\to\mathbb R$ be a definable function whose domain contains the graph of $(\gamma,\delta)$ restricted to  $(0,a)$ for some $a>0$.
	By \cite[Cell Decomposition Theorem (2.11)]{vdD} and \cite[Theorem 6.7]{Cos}, there exists a $C^1$ cell decomposition $\mathcal C$ of $\mathbb R^5$ such that, for any cell
	$C\in\mathcal C$, the restriction $g_{|C}$ of $g$ to $C$ is $C^1$. Since there are finitely many cells in $\mathcal C$, the curve
	$x\mapsto(x,\gamma,\delta)$ meets at least one cell $C$ infinitely many times as $x$ goes to $0$. Denote by $\mathds 1_C:\mathbb
	R^5\to\mathbb R$ the {characteristic} function of $C$. Since $1-\mathds 1_C$ is a definable function and $\gamma$, $\delta$ are not
	interlaced, by Proposition \ref{prop:sign}, the function $x\mapsto 1-\mathds 1_C(x,\gamma(x),\delta(x))$ has an ultimate sign as $x$ goes to zero. Since $(x,\gamma(x),\delta(x))$ meets $C$ infinitely many times, $1-\mathds 1_C(x,\gamma(x),\delta(x))$ vanishes
	infinitely many times, then identically. Hence, the graph of $(\gamma,\delta)$ is contained in $C$ for small $x$.
	
	Now, $x\mapsto (x,\gamma,\delta)$ is a $C^1$ curve, whose image is included in $C$ and the restriction of $g$ to $C$ is $C^1$. So the composition $g\circ (x,\gamma,\delta)$ is $C^1$. Moreover, its derivative is the composition of the differential of $g_{|C}$ with the derivative of $x\mapsto(x,\gamma,\delta)$. Since $\gamma,\delta$ are solutions of \eqref{SF}, this derivative is itself the composition of a definable function with $(x,\gamma,\delta)$. Summarizing, $g(x,\gamma,\delta)$ is ultimately $C^1$ and the germ of its derivative belongs to $\mathcal{F}(x,\gamma,\delta)$, which concludes the proof.
\end{proof}

\begin{proof}[Proof of Proposition \ref{prop:sign}]
	
	Let \eqref{SF} be a definable system  and $\gamma$, $\delta$ be solutions at 0 of \eqref{SF} that satisfy the hypothesis of Theorem \ref{th:main}. We prove Proposition \ref{prop:sign} by showing that if a germ in $\mathcal F(x,\gamma,\delta)$ does not have constant sign, then $\gamma, \delta$ are interlaced. Let $g:\R^5\to \R$ be definable in $\RR$, and whose domain contains the graph of $(\gamma,\delta)$. Set $G(x):=g(x,\g_1(x),\g_2(x),\delta_1(x),\delta_2(x))$. We shall conclude that $\gamma, \delta$ are interlaced if $G$ does not have an ultimate sign as $x$ goes to $0$.
	
	Denote as before $\varepsilon(x)=(\eps_1(x),\eps_2(x)):=\delta(x)-\gamma(x)$. Define
	\[
	h(x,y_1,y_2,z_1,z_2)=g(x,y_1,y_2,y_1+z_1,y_2+z_2),
	 \]
	so that $h$ is a definable function whose domain contains the graph of $(\gamma,\varepsilon)$, and $h(x,\gamma(x),\varepsilon(x))=G(x)$.
	Note that $(\gamma,\varepsilon)$ is a solution at $0$ of the differential system $\eqref{tildeSF}$:
\begin{equation}\label{tildeSF}\tag{$\tilde{S}_F$}
\left\{\begin{aligned}
		{y_1}'&=&f_1(x,y_1,y_2)\\
		{y_2}'&=&f_2(x,y_1,y_2)\\
		{z_1}'&=&f_3(x,y_1,y_2,z_1,z_2)\\
		{z_2}'&=&f_4(x,y_1,y_2,z_1,z_2)
		\end{aligned}\right. 
	\end{equation} 
	where $f_3$ and $f_4$ are the definable functions given by 
	\begin{align*}
		f_3(x,y_1,y_2,z_1,z_2):=f_1(x,y_1+z_1,y_2+z_2)-f_1(x,y_1,y_2);\notag \\ f_4(x,y_1,y_2,z_1,z_2):=f_2(x,y_1+z_1,y_2+z_2)-f_2(x,y_1,y_2).\notag\end{align*}
	We set $\tilde F =(f_1,f_2,f_3,f_4)$ and  $c:\R\rightarrow \R^5,\ x\mapsto (x,\gamma(x),\varepsilon(x))$.
	
	Up to extending $h$ by zero, we assume that its domain is $\mathbb R^5$. {By the   already cited  Cell Decomposition Theorem}, there exists a cell decomposition $\mathcal{C}$ of $\R^5$ into cells of class $C^1$ adapted to the {signs} of $z_1$, $z_2$ and $h$. This means that, for each cell $C$ in $\mathcal C$, the restrictions $z_1{}_{|C}$, $z_2{}_{|C}$ and $h_{|C}$ have constant sign.
	
	In the following claims, we describe the intersection of the cells of $\mathcal C$ with the image of $c$. Denote by $\pi_0, \pi_1$ the linear projections given by
	\begin{equation*}\begin{aligned}
		\pi_0:\R^5\to\R^3,\;\;\pi_0:(x,y_1,y_2,z_1,z_2)\mapsto (x,y_1,y_2), \\
		\pi_1:\R^5\to\R^4,\;\;\pi_1:(x,y_1,y_2,z_1,z_2)\mapsto(x,y_1,y_2,z_1).
	\end{aligned}
\end{equation*}
	We introduce $\mathcal A$ and $\mathcal B$, respectively the cell decompositions of $\mathbb R^3$ and $\mathbb R^4$ induced by $\mathcal C$: the cells of $\mathcal A$ and $\mathcal B$ are respectively the images $\pi_0(C)$ and $\pi_1(C)$ of cells $C\in\mathcal C$.
	
	\begin{claim}\label{cla:3cell}
		There exists a unique $A\in\mathcal A$ such that $\pi_0(c(x))\in A$ for small $x$.
	\end{claim}
	{\em Proof.}
	$\mathcal A$ is a finite partition of $\mathbb R^3$,  then at least one cell $A\in\mathcal A$ contains infinitely many points of the curve $x\mapsto \pi_0(c(x))$  when $x$ goes to $0$. Since $\pi_0(c(x))=(x,\gamma(x))$ and $\gamma$ has the regular separation property, Lemma \ref{lem:1} shows that $\pi_0(c(x))$ belongs to $A$ for small $x$. Finally, $\mathcal A$ is a partition, so $x\mapsto \pi_0(c(x))$ cannot ultimately meet any other cell in $\mathcal A$ but $A$, so $A$ is unique.
	\qed
	
	\strut
	
	The decomposition of $A\times \mathbb{R}\subset \R^4$ induced by $\mathcal{B}$ is given by:
	\begin{itemize}
		\item finitely many graphs of definable maps over $A$, totally ordered, including the null function (since the partition is adapted to $z_1=0$). We denote these functions by $\phi_{-r}<\dots<\phi_{-1}<\phi_0:=0<\phi_1<\dots<\phi_{\ell}$;
		\item the strips between these graphs, given and denoted by
		\[ 
		(\phi_{j-1},\phi_j):=\{(x,y_1,y_2,z_1);\;(x,y_1,y_2)\in A,
		\phi_{j-1}(x,y_1,y_2)<z_1<\phi_{j}(x,y_1,y_2)\},
		\] for $j=-r,\dots,\ell+1$, where we put $\phi_{-r-1}\equiv -\infty$ and $\phi_{\ell+1}\equiv+\infty$.
	\end{itemize}
	
	We set $B^-:=(\phi_{-1},0)$, $B^0:=A\times\{0\}$ and $B^+:=(0,\phi_{1})$.
	
	\begin{claim}\label{cla:4cell}
		For small $x$, $\pi_1(c(x))\in B^-\cup B^0\cup B^+$.
	\end{claim}
	
	{\em Proof.} Since $\gamma$ has regular separation property and $\phi_{-1}$ and $\phi_1$ do not vanish over $A$, which contains the graph of $\gamma$ by Claim \ref{cla:3cell}, there exist natural numbers $n,\, m$ such that, ultimately, $|\phi_{-1}(x,\gamma(x))|>x^{n_{}}$ and $|\phi_1(x,\gamma(x))|>x^{m}$. On the other hand, since $\gamma$ and $\delta$ have flat contact, $|\varepsilon_1(x)|<x^{\max(n_{},m)}$ for small $x$. So, $\phi_{-1}(x,\gamma(x))<\varepsilon_1(x)<\phi_{1}(x,\gamma(x))$ for small $x$, qed.
	\qed
	
	\strut
	
	For $\alpha\in\{-,0,+\}$, the decomposition of $B^\alpha\times\R$ induced by $\mathcal{C}$ {consists} of finitely many graphs of definable functions
	$\psi_{-r_{\alpha}}^\alpha<\dots<\psi^{\alpha}_0:=0<\dots<\psi_{\ell_{\alpha}}^\alpha$ over $B^\alpha\subset \R^4$ and  strips $(\psi^\alpha_{j-1},\psi^\alpha_{j})$ for $j=-r_{\alpha},...,\ell_{\alpha}+1$ (where, as above, $\psi_{-r_{\alpha}-1}^\alpha := -\infty$ and $\psi_{\ell_{\alpha}+1}^\alpha :=+\infty$).
	We denote by $C^-=(\psi^0_{-1},0)$ and $C^+=(0,\psi^0_{1})$ the two strips over $B^0$ adjacent to $B^0\times \{0\}=A\times{(0,0)}$. As  in Claim \ref{cla:4cell}, we get that:
	\begin{claim}\label{cla:5cell}
		For small $x$, $\pi_1(c(x))\in B^0\Rightarrow c(x)\in C^+\cup C^-$.
	\end{claim}
	{\em Proof.}
	Following the same proof as  in Claim \ref{cla:4cell}, the regular separation of $\gamma$ and the flatness of $\varepsilon_2$ give that $\psi^0_{-1}(x,\gamma(x),0)< \varepsilon_2(x)<\psi^0_{1}(x,\gamma(x),0)$ for small $x$. Now, since $\gamma$ and $\delta$ are different solutions of \eqref{SF}, $\varepsilon(x)\neq (0,0)$ for small $x$. Hence, ultimately, if $\pi_1(c(x))\in B^0$, then $\varepsilon_1(x)=0$, so $\varepsilon_2(x)\neq 0$. We obtain that $c(x)\notin B^0\times\{0\}$, so $c(x)\in C^+\cup C^-$ as required.
	\qed

	\begin{claim}\label{lm:sign-H}
		If $G$ has not an ultimate sign, then the parameterized curve $x\mapsto\pi_1(c(x))$ meets $B^0$ infinitely many times as $x$ goes to $0$.
	\end{claim}
	
	\begin{proof}
		We proceed by contradiction. 
		Suppose that $\pi_1(c(x))$ does not meet $B^0$ infinitely many times as $x$ goes to $0$. Then, $\varepsilon_1$ has an ultimate sign,  say $\varepsilon_1(x)>0$ for small $x$ (the case $\varepsilon_1(x)<0$ is similar). So, from Claim \ref{cla:4cell}, $c(x)\in B^+\times \mathbb R$ ultimately.
		
		Note that $h$ has a constant sign on each connected component of the union \[ \displaystyle U=\bigcup_{j=r_+-1}^{\ell_+} (\psi^+_j,\psi^+_{j+1})\] of the strips of $\mathcal C$ over $B^+$. Since  $G=h\circ c$ has not an ultimate sign and the curve $c$ is continuous, its image meets  infinitely many times the complement of $U$ in $B^+\times \mathbb R$, that is,
		the graphs of the $\psi^+_j$. 
		Hence, there is an index $k$ such that the curve $c$ meets  infinitely many times the graph of $\psi^+_k$.
		
		Let $w(x,y_1,y_2,z_1,z_2):=z_2-\psi_k^+(x,y_1,y_2,z_1)$, so $w\circ c(x)$ vanishes whenever $c(x)$ belongs to the graph $\psi_k^+$. We first prove that ultimately, the sign of $(w\circ c)'(x)$ is constant, in fact equal to 0,  when $w(c(x))=0$. For this, remark that \[ (w\circ c)'(x)=dw(c(x))(c'(x)) = dw(c(x))\left(1,\tilde{F}(c(x))\right),\]
		and set \[ S(x,y_1,y_2,z_1)= dw(x,y_1,y_2,z_1,\psi^+_k(x,y_1,y_2,z_1))\left(1,\tilde{F}(x,y_1,y_2,z_1,\psi^+_k(x,y_1,y_2,z_1))\right).\]
		Then $S$ is definable, so, by Lemma \ref{lem:2}, there exists a sign $\alpha\in\{-,0,+\}$ and $m>0$ such that, ultimately,
			\[ 0<z_1<x^m\Rightarrow \mathrm{sign}(S(x,\gamma(x),z_1))=\alpha.\]
		In particular, for small $x$, $0<\varepsilon_1(x)<x^m$. So, if $w(c(x))=0$, then \[ \mathrm{sign}((w\circ c)'(x))=\mathrm{sign}(S(x,\gamma(x),\varepsilon_1(x)))=\alpha.\]
		
		This is impossible if $\alpha\neq 0$: the differentiable map $w\circ c$ vanishes infinitely many times, but its derivative would  have a constant nonzero sign when $w\circ c(x)=0$, contradicting Proposition \ref{prop:0}.
		So $\alpha=0$, i.e. $(w\circ c)'(x)=0$, whenever $w\circ c(x)=0$ for small $x$. 
		
		We get from equation (1) and system \eqref{tildeSF} 
		 that, for small $x$ and $0<z_1<x^m$:
		\[f_4(v)=d\psi_k^+(u)(1,f_1(v),f_2(v),f_3(v)) \textrm{ with } u:=(x,\gamma(x),z_1),\ v:=(u,\psi_k^+(u))\]
		But this leads again to a contradiction. Indeed, let $\zeta(x):=\psi_k^+(x,\gamma(x),\varepsilon_1(x))$. Then
		$\zeta'(x)=d\psi_k^+(x,\gamma(x),\varepsilon_1(x))(1,\gamma'(x),\varepsilon_1'(x))=f_4(x,\gamma(x),\varepsilon_1(x),\zeta(x))$, so
		$\zeta$ and $\varepsilon_2$ satisfy the same differential equation. By uniqueness of solutions,  $\zeta$ and $\varepsilon_2$ coincide since $\zeta(x)=\varepsilon_2(x)$ whenever $c(x)$ belongs to the graph of $\psi_k^+$. This means that $c(x)$ is included in the graph of $\psi_k^+$, which contradicts the fact that $h\circ c$ has no ultimate sign (recall that $h$ has constant sign over the graph of $\psi_k^+$). Claim 3.8 is proven.
	\end{proof}
	
	We now have all elements to finish the proof of Proposition \ref{prop:sign}. Let $\Theta(x):=\frac{\varepsilon(x)}{||\varepsilon(x)||}\in\mathbb S^1$. Choose a continuous function $\theta:(0,a)\to\mathbb R$ such that $\Theta(x)=(\cos \theta(x),\sin \theta(x))$. Remark that $\theta$ is in fact $C^1$ since $\Theta$ is.
	\begin{claim}\label{cla:angle}
		If $G$ has not an ultimate sign, the angle $\theta(x)$ diverges to infinity as $x$ {tends} to $0$.
	\end{claim}
	\begin{proof}
		From Claim \ref{lm:sign-H}, if $G$ has not an ultimate sign, the function  $x\mapsto\varepsilon_1(x)$ vanishes infinitely many times as $x$ goes to $0$. From Claim \ref{cla:5cell}, $\varepsilon_1(x)=0 \Rightarrow c(x)\in C^+\cup C^-$, so at least one of these two cells, say $C^+$ (the proof for $C^-$ being analogous), {intersects} infinitely many times the curve $x\mapsto c(x)$ as $x$ {tends} to $0$.
		
		We first show that, for small $x$, the derivative $\theta'(x)$ has a constant nonzero sign when $c(x)\in C^+$.
		Indeed, if $\varepsilon_1(x)=0$ and $\varepsilon_2(x)>0$, then $\theta(x)=\frac{\pi}{2}\text{ mod }2\pi$, so $\theta'(x)$ has a   sign opposite to the one of $\varepsilon_1'(x)$, because $\varepsilon_1'(x)=-||\varepsilon(x)||\theta'(x)$ at such a value. So  $\theta'(x)$ has a sign opposite to that of
		$f_3(x,\gamma(x),0,\varepsilon_2(x))$ when $\varepsilon_1(x)=0$. We apply Lemma \ref{lem:2} to  the map \[(x,y_1,y_2,z_2)\mapsto f_3(x,y_1,y_2,0,z_2).\] There exists $m>0$ and a sign $\alpha\in\{-,0,+\}$ such that, for small $x$,
		 \[0<z_2<x^{m}\Rightarrow\mathrm{sign}(f_3(x,\gamma(x),0,z_2))=\alpha.\]
		For small $x$, $|\varepsilon_2(x)|<x^m$, so, $\mathrm{sign}(\theta'(x))=-\alpha$ whenever $c(x)\in C^+$.
		Moreover, $\alpha\neq 0$. Otherwise, the null function would be a solution of the same differential equation $z_1'=f_3(x,\gamma(x),z_1,\varepsilon_2(x))$ as the one for $\varepsilon_1$, and, locally at any value $x$ for which  $\varepsilon_1(x)=0$, this solution coincides with $\varepsilon_1$, then everywhere by uniqueness. It would imply that $c(x)\in C^+$ ultimately, while $h$ has constant sign on $C^+$, a contradiction.
		
		For simplicity, suppose $\alpha=+$ (the proof for $\alpha=-$ is similar). Since $\theta'(x)\neq 0$ if $\theta(x)=\frac{\pi}{2}\text{ mod }2\pi$, $\theta^{-1}(\frac{\pi}{2}+2\pi\mathbb Z)$ is discrete. Let $(x_n)\to 0$ be the infinite decreasing sequence formed by the elements of $\theta^{-1}(\frac{\pi}{2}+2\pi\mathbb Z)$. By the definition of this sequence $(x_n)$ and the continuity of $\theta$, we have that $\theta(x_{n+1})-\theta(x_n)=2\pi s_n$ with $s_n\in\{-1,0,1\}$. Since $\theta'(x_{n+1})<0$, we cannot have $s_n=-1$, and, since  $\theta'(x_{n+1})\theta'(x_n)> 0$, we cannot have $s_n=0$ either.
		Hence, $\theta(x_{n+1})=\theta(x_n)+2\pi$ for all $n$, and $\lim_{n\to\infty} \theta(x_n)=+\infty$. Moreover, since $\theta(x_{n+1})>\theta(x)> \theta(x_n)$ for any $x\in(x_{n+1},x_n)$, we get: \[\lim_{x\to 0} \theta(x) = \lim_{n\to\infty}\theta(x_n)=+\infty,\]
		which ends the proof.
	\end{proof}
	
	{To conclude the proof of Proposition \ref{prop:sign}, consider $\gamma$ and $\delta$ which satisfy the  hypothesis.} Then, either any germ in $\mathcal F(x,\gamma,\delta)$ has an ultimate sign, {which is the first alternative of the proposition}, or there exists $G\in \mathcal F(x,\gamma,\delta)$ which has not an ultimate sign. Then, Claim \ref{cla:angle} shows that any continuous measure $\theta(x)$ of the angle between $\Theta(x)$ and $(1,0)$ diverges to $\pm\infty$ as $x$ {tends to} $0$. In other words, $\gamma$ and $\delta$ are interlaced, {that is to say, the second alternative of the proposition holds}.
\end{proof}

\section{Dichotomy ``interlaced or separated'' for trajectories of definable vector fields.}\label{sec:trajectories}

We consider a  vector field $\xi$ of class $C^1$ in a neighborhood of $0\in\R^3$, definable in some polynomially bounded o-minimal structure $\mathcal{R}$, and such that $\xi(0)=0$.
A \emph{trajectory} {(at zero)} of $\xi$ is the image $\Gamma=c((0,a))$ of an \emph{integral curve} $c$ of $\xi$, that is, of a parameterized $C^1$ curve $c:(0,a)\to\R^3\setminus\{0\}$, $a >0$, with $\lim_{t\to 0^{+}}c(t)=0$, $(\xi(c))^{-1}(0)=\emptyset$ and $\xi(c) \wedge c' = 0$. We {might also denote $\Ga$ by $|c|$} and say that $c$ is a parameterization of $\Ga$.
We are only interested in the behavior of $\Gamma$ near $0$, and we identify trajectories as soon as they have the same germ at $0$.
\subsection{Adapted charts and regular separation}

A \emph{definable chart} $C=(V,X)$ at $0$ (a chart for short) consists of an open definable set $V\subset \mathbb R^3$ such that  $0\in\overline{V}$, and a definable mapping $X : V\to \mathbb R^3$ {which} is a diffeomorphism onto its image such that $\lim_{V\ni p\to 0}X(p) = 0$.
A chart $C=(V,(x,y_1,y_2))$ is said to be \emph{adapted to {a} trajectory $\Gamma$} if $V$ contains a representative of the germ at $0$ of $\Gamma$ and the restriction of $dx(\xi)$ to this representative is positive. 
In this situation, $\Gamma \cap V$ can be parameterized by $x$: by the Inverse Function Theorem, there is a unique parameterization $c:(0,a)\to V$ of $\Gamma\cap V$ such that $x(c(t))=t$.
We denote $\gamma_C(x)=(\gamma_{C,1}(x),\gamma_{C,2}(x)):=(y_1(c(x)),y_2(c(x)))$. Accordingly, if $C$ is adapted to {a} trajectory $\Delta$, we define $\delta_C(x)$ by $(x,\delta_C(x))\in \Delta$, and $\lim_{x\to 0}\delta_C(x)=0$. Note that, if $(V,(x,y_1,y_2))$ is adapted to $\Gamma$, then so is any other definable chart of the form $(V,(x,z_1,z_2))$.
For $p\in V\setminus (dx(\xi))^{-1}(0)$, set
\[F_C(p):=\left(f_{1C}(p) := \frac{dy_1(\xi)}{dx(\xi)}(p), f_{2C}(p):=\frac{dy_2(\xi)}{dx(\xi)}(p)\right).\]
So $\gamma_C$ and $\delta_C$ are both solutions of the differential system $(S_{F_C})$ in the sense of the previous sections.

\begin{definition}\label{def:separatedtrajectories}~
	\begin{itemize} \item {Let $\Gamma$  be a trajectory of $\xi$. We will say that} $\Gamma$ has the \emph{regular separation property} if there exists a definable chart $C$ adapted to $\Gamma$ such that:
		\begin{itemize}
			\item[(i)] The map $\gamma_C$ has the regular separation property in the sense of Definition \ref{def:regular-flat};
			\item[(ii)] For any definable function $f:\mathbb R^3\to\mathbb R$, if $\lim_{x\to 0}f(x,\gamma_C(x))= 0$, then $\exists k\in\mathbb N,\; \exists a>0,$ such that $|f(x,\gamma_C(x))|<x^{\frac{1}{k}}$ for all $x\in(0,a)$.
		\end{itemize}
		\item {Let $\Gamma$, $\Delta$ be trajectories of $\xi$. We will say that} $\Gamma,\Delta$ have \emph{flat contact} if they admit a common adapted chart $C$ such that $\gamma_C$ and $\delta_C$ have flat contact in the sense of Definition \ref{def:regular-flat}.\end{itemize}
\end{definition}

{Notice that our definition of flat contact requires the existence of a common adapted chart.} {This condition is not always fulfilled, as shown by \cite[Example 21]{San},
	which gives a family of trajectories of an analytic vector field, all included in a common ``flat horn'', but none of these trajectories can be parametrized by any analytic coordinate}. 
Condition (ii) above {might seem superfluous, since it} is not needed to apply Theorem \ref{th:main}. However, we will see that (ii) insures that $\gamma_{C}$ has the regular separation property independently of the chart $C$. Moreover, the flat contact between $\gamma_C$ and $\delta_C$ will be shown to be independent of the adapted chart $C$ if $\Gamma$ has the regular separation property.
The need for condition (ii) appears clearly in the {following example}. 

\begin{example}
	{\em	Consider the vector field
		\[ \xi =x^2\frac{\partial}{\partial x}+y^2 x\frac{\partial}{\partial y}+z\frac{\partial}{\partial z}.\]
		If $\Gamma$ is a trajectory of $\xi$ at the origin contained in $\{x>0\}\cap\{y>0\}$, 
		then the charts $C=(\R^3,(x,y,z))$ and $C'=(\R^3,(y,x,z))$ are both adapted to $\Gamma$.
		{One} can describe {these} trajectories as a family $(\Gamma_{\alpha,\beta})$ indexed by parameters $(\alpha,\beta)\in \mathbb R^*_+\times\mathbb R$, where
		$\Gamma_{\alpha,\beta}$ has respective parameterizations in the charts $C$ and $C'$ given by:
		 \begin{align*}
			(x,\gamma_C^{\alpha,\beta}(x)) 
			 =  \left(x,\left(\log\frac{\alpha}{x}\right)^{-1}, \beta e^{-\frac{1}{x}}\right),\\
			(y,\gamma_{C'}^{\alpha,\beta}(y))
			= \left(y,\alpha e^{-\frac{1}{y}}, \beta e^{-\frac{1}{\alpha} e^{\frac{1}{y}}}\right).
		\end{align*}
		
		For any $\alpha>0$, the map $\gamma_C^{\alpha,0}$ satisfies the {regular separation property} 
		(Definition \ref{def:regular-flat}), so 
		$\Gamma_{\alpha,0}$ satisfies the condition (i) for regular separation of trajectories  (Definition \ref{def:separatedtrajectories}). But the condition (ii) is not satisfied by $\gamma_C^{\alpha,0}$  (with $f(x,y,z)=y$ for instance).
		This does not {yet show that $\Gamma_{\alpha,0}$ has not} regular separation property 
		({since} Definition \ref{def:separatedtrajectories} {only requires} one chart), but it will follow from 
		Proposition \ref{propo:reg-sep-cdv} below.
		
		{In} the chart $C'$, {note} that none among the functions $\gamma_{C'}^{\alpha,0}$ has regular separation property
		since their first coordinate is neither zero nor bounded {from} below by a power of $y$. {This} 
		shows that condition (i) might depend on the chart $C$.
		
		{The present example also shows} that flat contact between $\gamma_C$ and $\delta_C$ 
		depends on the chart $C$. 
		{In} the chart $C'$, if $\alpha\neq\alpha'$, the maps  $\gamma^{\alpha,\beta}_{C'}$ and $\gamma^{\alpha',\beta'}_{C'}$ 
		have flat contact, while the corresponding $\gamma^{\alpha,\beta}_C$ and $\gamma^{\alpha',\beta'}_C$
		do not have flat contact (the difference {between} their first coordinate is not flat with respect to $x$). 
		However, the trajectories $\Gamma_{\alpha,\beta}$ and $\Gamma_{\alpha',\beta'}$ have 
		flat contact according to our definition \ref{def:separatedtrajectories}, since their parameterizations have flat contact in at least one chart.}
\end{example}

\begin{proposition}\label{propo:reg-sep-cdv}
	{Assume that $\Gamma$ has regular separation property. If}  $C$ is a chart adapted to $\Gamma$, then $\gamma_{C}$ satisfies (i) and (ii) of Definition \ref{def:separatedtrajectories}.
\end{proposition}
\begin{proof}
	Let $C=(V,X=(x,y_1,y_2))$ be {a chart adapted to} $\Gamma$ and $C'=(V',X'=(x',y_1',y_2'))$ be a chart satisfying the conditions (i) and (ii) of Definition \ref{def:separatedtrajectories}. Let $f$ be a definable function whose domain contains $\Gamma$, $f_C=f\circ X^{-1}$, $f_{C'}=f\circ X'^{-1}$. {Note} that
	$f_C(t,\gamma_C(t))=f_{C'}(t',\gamma_{C'}(t'))$ with $t'=x'(X^{-1}(t,\gamma_C(t)))$, or equivalently, $t = x( X'^{-1}(t',\gamma_{C'}(t')))$. So $t$ and $t'$ are simultaneously positive and $t$ {tends to} $0$ if and only if $t'$ {does so}.
	
	In particular, if $f_{C'}(t',\gamma_{C'}(t'))=0$ for small $t'$, then $f_C(t,\gamma_C(t))=0$ for small $t$.
	Otherwise, from condition (i), there is an $n$ such that $|f_C(t,\gamma_C(t))|> t'^n$ for small $t'$.
	But, from condition (ii), since $x\circ X'^{-1}$ is a definable function whose domain contains $\Gamma$, there is a $k$ such that
	$t < t'^{1/k}$, so $t'>t^k$. Finally,
	if $f_C(t,\gamma_C(t))$ doesn't vanish identically, $|f_C(t,\gamma_C(t))|> t^{kn}$, so $\gamma_C$ has regular separation {property} (Definition \ref{def:regular-flat}).
	
	Moreover, by (i), there exists $N>0$ such that $t>t'^N$. If  $\lim_{t\to 0}f(t,\gamma_C(t))= 0$, there exists $M>0$ such that $|f(t,\gamma_C(t))|<t'^{\frac{1}{M}}$, again by (ii). We {obtain} that $|f(t,\gamma_C(t))|<t^{\frac{1}{NM}}$ which shows property (ii) for $\gamma_C$. 
\end{proof}

\subsection{Flat contact and interlacement}
\begin{proposition}\label{prop:flat}
	{Assume that $\Gamma$ has regular separation property and that $\Delta$ has flat contact with $\Gamma$.} Then, for any chart $(V,T)$ adapted to $\Gamma$, there exists $V'\subset V$ such that $C':=(V',T)$ is adapted to $\Gamma$ and $\Delta$, and $\gamma_{C'}$ and  $\delta_{C'}$ have flat contact.
\end{proposition}

\begin{proof}
	Let $C=(W, X=(x,y_1,y_2))$ be a chart adapted to both $\Gamma$ and $\Delta$, such that $\gamma_C$ and $\delta_C$ have flat contact. Let $(V,T=(t,z_1,z_2))$ be a chart adapted to $\Gamma$. We may assume that  $V$ is bounded.
	
	For $p\in V\cap W$ and $r>0$, we set
	\[ 
		B_C(p,r)=\{q\in V\cap W;\; ||X(q)-X(p)|| < r\}.
	\]
	Let $V'=\{p\in V\cap W;\; dt(\xi)(p)>0\}$, and $r(p)=\sup\{r>0, B_C(p,r)\subset V'\}$. Then $r$ is definable and positive over $\Gamma\cap V\cap W$, so, by Proposition \ref{propo:reg-sep-cdv}, $r(X^{-1}(x,\gamma_C(x)))>x^N$ for some $N>0$ and for all sufficiently small $x$. On the other hand, by flat contact, $||\delta_C(x)-\gamma_C(x)||<x^N$ for sufficiently small $x$. So, $X^{-1}(x,\delta_C(x))$ ultimately belongs to $V'$, which proves that $V'$ contains a representative of the germ at $0$ of $\Delta$. Hence, $C'=(V',(t,z_1,z_2))$ is adapted to $\Delta$.

	Let $n\ge 1$. Let us show that $||\gamma_{C'}(t)-\delta_{C'}(t)||/t^n$ is bounded for small $t$.
	Given $t$, we set $h_t=T^{-1}(t,\delta_{C'}(t))\in \Delta$, $X(h_t)=(x,\delta_{C}(x))$ and  $g_x=X^{-1}(x,\gamma_C(x))\in\Gamma$,  $g_t=T^{-1}(t,\gamma_{C'}(t))\in \Gamma$. Moreover, we set $t'=t(g_x)$ (so $T(g_x)=(t',\gamma_{C'}(t')))$. Note that $x$, $g_x$, $g_t$, $t'$ all depend on the given data $t$. 
	In what follows, we let $t$, $t'$ and $x$ vary, but keeping the previous relations between them. 
	
	We have
	$||\gamma_{C'}(t)-\delta_{C'}(t)||=||T(g_t)-T(h_t)|| \le ||T(g_t)-T(g_x)|| +||T(g_x)-T(h_t)||$. We get a  bound for the two terms
	$ ||T(g_t)-T(g_x)||$ and $||T(g_x)-T(h_t)||$
	independently.
	
	{\bf Bound for $||T(g_x)-T(h_t)||$.} For $p\in V'$, let \[ \rho(p) := \sup\{r>0;\; \forall q\in B_C(p,r),\ ||T(q)-T(p)||<t(q)^n\}.\]
	The map $\rho$ is definable, and positive over $\Gamma$, so, ultimately, $\rho(g_x)>x^m$ for some $m$  by regular separation of $\Gamma$, while $||X(g_x)-X(h_t)||< x^m$, i.e. $g_x\in B_C(h_t,x^m)$, by flat contact between $\gamma_C$ and $\delta_{C}$. Hence for small $x$, $||T(g_x)-T(h_t)||<t^n$,
	and, in particular, $||T(g_x)-T(h_t)||/t^n$ is bounded as $t$ tends to $0$.

	{\bf Bound for $ ||T(g_t)-T(g_x)||$.}
	Let
	\[ s := \displaystyle \sqrt{\frac{dz_1(\xi)^2+dz_2(\xi)^2}{dt(\xi)^2}}\]
	so $||{(\gamma_{C'}})'(t)||=s(g_t)$.
	The function $\frac{1}{s}$ is definable and positive over $\Gamma$, so, by regular separation and Proposition \ref{propo:reg-sep-cdv}, there exists
	$\alpha >0$ such that $\frac{1}{s(g_t)}>t^{\alpha}$, i.e. $s(g_t)<t^{-\alpha}$, for small $t$.
	We set
	\[ R(p) := \sup\{r>0;\; \forall q\in B_C(p,r),\ |t(p)-t(q)|<t(q)^{n+\alpha}\}.\]
	Again, the map $R$ is definable, and positive over $\Gamma$, so, by regular separation,
	there exists $m>0$ such that $R(g_x)>x^m$ for small $x$, while, by flat contact, $||\delta_{C}(x)-\gamma_{C}(x)||<x^m$, i.e.  $||X(h_t)-X(g_x)||<R(g_x)$. In other words, $h_t\in B(g_x,R(g_x))$, 
	so, for small $x$ (hence small $t$),
	$|t(g_x)-t(h_t)|<t(h_t)^{n+\alpha}$, i.e. $|t'-t|<t^{n+\alpha}$. But	 
	\begin{align}
		||T(g_x)-T(g_t)||^2 & = & |t'-t|^2+||\gamma_{C'}(t')-\gamma_{C'}(t)||^2\notag \\
		& \le & |t'-t|^2+ \left(\sup_{\tau\in(t,t')} ||{(\gamma_{C'}})'(\tau)||\cdot |t'-t|\right)^2\notag \\
		& \le & \left(1+ \max(t^{-\alpha},t'^{-\alpha})^2\right)|t'-t|^2.\notag
	\end{align}	
	For small $t$, $|t-t'|<t^{n+\alpha}$ and $1+\max(t^{-\alpha},t'^{-\alpha})^2<Mt^{-2\alpha}$ for some $M\in\mathbb R$, because $|t'-t|<t^{n+\alpha}$ also implies $t/t'$ bounded. So we finally obtain that
	\[ ||T(g_x)-T(g_t)||\le Mt^{-\alpha}t^{n+\alpha}\le Mt^n.\]
	In particular, $||T(g_x)-T(g_t)||/t^n$ is bounded as $t$ tends to $0$, which {concludes} the proof that $\gamma_{C'}$ and $\delta_{C'}$ have flat contact.
\end{proof}

\begin{proposition}\label{prop:interlacement}
	{Assume that $\Gamma$ has regular separation property and that $\Delta$ has flat contact with $\Gamma$.} Then the property of interlacement of $\gamma_C$ and $\delta_C$ does not depend on the choice of a chart $C$ adapted to $\Gamma$.
\end{proposition}
Our proof is based on the two following claims, similar to Lemmas 1.7 and 1.8 in \cite{Can-Mou-San2}. For both of them, we fix a chart $C=(V,(x,y_1,y_2))$ adapted to $\Gamma$, which we assume -- according to Proposition \ref{prop:flat} 
-- to be also adapted to $\Delta$.  
\begin{claim}\label{claim:intx} Let $C'=(V',(x,z_1,z_2))$ be another chart adapted to both $\Gamma$ and  $\Delta$. If  $\gamma_C$ and $\delta_C$ are interlaced, then so are $\gamma_{C'}$ and $\delta_{C'}$.
\end{claim}
\begin{proof} 
	Set $Y=(x,y_1,y_2)$ and $Z=(x,z_1,z_2)$.
	Since $\gamma_C$ and $\delta_C$ are interlaced, there exists an infinite decreasing sequence $(t_n)$ {tending to} $0$ such that $\gamma_{C,2}(t_{2n})-\delta_{C,2}(t_{2n})=0$ (every time $\theta\equiv 0\,\mathrm{mod }[\pi]$ in Definition \ref{def:interlaced-separated}) while $\gamma_{C,2}(t_{2n+1})-\delta_{C,2}(t_{2n+1})\neq 0$  (e.g. every time $\theta\equiv \pi/2 \,\mathrm{mod }[\pi]$). Set
	\[D(x,z_1,z_2,z'_1,z'_2)=y_2(Z^{-1}(x,z_1,z_2))-y_2(Z^{-1}(x,z'_1,z'_2)).\] 
	Then $\phi:t\mapsto D(t,\gamma_{C'}(t),\delta_{C'}(t))$ vanishes over
	$(t_{2n})_{n\in\mathbb N}$, but vanishes at no point of $(t_{2n+1})_{n\in\mathbb N}$. Since $\phi$ is the composition of a definable function with \\ $t\mapsto(t,\gamma_{C'}(t),\delta_{C'}(t))$, this implies that the ring $\mathcal F(x,\gamma_{C'}(x),\delta_{C'}(x))$ is not a Hardy field. So {by} Theorem \ref{th:main} and Propositions \ref{propo:reg-sep-cdv} and \ref{prop:flat}, $\gamma_{C'},\delta_{C'}$ are interlaced.
\end{proof}

\begin{claim}\label{claim:intxy} Suppose that $C'=(V,(y_1,x,y_2))$ is adapted to $\Gamma$ and $\Delta$.  If $\gamma_C$ and $\delta_C$ are interlaced, then so are $\gamma_{C'}$ and $\delta_{C'}$.
\end{claim}
\begin{proof}
	Since $\gamma_C$ and $\delta_C$ are interlaced, there is an infinite decreasing sequence $(x_n)_{n\in\mathbb N}$ tending to $0$ such that
	\[ \gamma_{C,1}(x_{2n})=\delta_{C,1}(x_{2n}) \text{ and }
	\gamma_{C,1}(x_{2n+1})\neq\delta_{C,1}(x_{2n+1}).\]
	Let $t_n:=\gamma_{C,1}(x_{n})$, so $x_n=\gamma_{C',1}(t_n)$.
	Then $x_{2n}=\gamma_{C',1}(t_{2n})=\delta_{C',1}(t_{2n})$,
	while $x_{2n+1}=\gamma_{C',1}(t_{2n+1})\neq\delta_{C',1}(t_{2n+1})$:
	otherwise, $\delta_{C,1}(x_{2n+1})=t_{2n+1}=\gamma_{C,1}(x_{2n+1})$,
	which is false.
	Then $\gamma_{C',1}-\delta_{C',1}$ vanishes over the sequence $(t_{2n})_{n\in\mathbb N}$ that tends to $0$,
	while it vanishes at no point of the sequence $(t_{2n+1})_{n\in\mathbb N}$,
	that tends to $0$ too.
	Then $\mathcal F(x,\gamma_{C'}(x),\delta_{C'}(x))$ is not a Hardy field,
	so from Theorem \ref{th:main}, $\gamma_{C'}$ and $\delta_{C'}$ are interlaced.
\end{proof}

\begin{proof}[Proof of Proposition \ref{prop:interlacement}]
	Let $C_0=(V_0,(x,y_1,y_2))$ be a chart adapted to $\Gamma$ such that $\gamma_{C_0}$ and $\delta_{C_0}$ are interlaced, and let $C_1=(V_1,(t,z_1,z_2))$ be another chart adapted to $\Gamma$. There is a chart $C_2=(V_2,(x,y'_1,y'_2))$ that is adapted to $\Gamma$ and $\Delta$ and such that $dy'_1(\xi)>0$, $dy'_2(\xi)>0$: e.g., set  $y'_i= -y_i $ if $ dy_i(\xi)<0$, or $y'_i= y_i+ x$ if $ dy_i(\xi)\equiv 0$. Similarly, there is a chart $C_3=(V_3,(t,z'_1,z'_2))$, adapted to $\Gamma$ and such that  $dz'_1(\xi)>0$, $dz'_2(\xi)>0$. We may assume that $z_1',z_2'$ are chosen so that $C_4=(V_4,(x,z_1',z_2'))$ is a chart, adapted to $\Ga$ (for instance, we change $z_1'$ to $z_1'+\alpha x+\beta y_1'+\gamma y_2'$ with generic $\alpha,\beta,\gamma>0$ and use the fact that $(x,y_1',y_2')$ is a chart). According to Proposition \ref{prop:flat}, up to considering smaller domains, all the charts above are also assumed to be adapted to $\Delta$.  
	
	Set $W=V_0\cap V_1 \cap V_2 \cap V_3 \cap V_4$, and restrict the domains of $C_0, C_1, C_2, C_3, C_4$ to $W$ (keeping the same names).
	From Claim \ref{claim:intx}, we get that $\gamma_{C_2}$ and $\delta_{C_2}$ are interlaced. 
	From Claim \ref{claim:intx}, $\gamma_{C_4}$ and $\delta_{C_4}$ are interlaced.
	From Claim \ref{claim:intxy}, $C_5=(W,(z'_1,x,z'_2))$ is a chart adapted to $\Ga$ and $\Delta$, and  $\gamma_{C_5}$ and $\delta_{C_5}$ are interlaced. 
	From Claim \ref{claim:intx},  $C_6=(W,(z'_1,t,z'_2))$ is also a chart adapted to  $\Ga$ and $\Delta$, and  $\gamma_{C_6}$ and $\delta_{C_6}$ are interlaced.
	From Claim \ref{claim:intxy},  $\gamma_{C_3}$ and $\delta_{C_3}$ are interlaced for  $C_3=(W,(t,z'_1,z'_2))$.
	And finally, from Claim \ref{claim:intx}, $\gamma_{C_1}$ and $\delta_{C_1}$ are interlaced, which was to be proven.
\end{proof}

\subsection{Dichotomy Interlacement versus Separation by projection}

\begin{definition}\label{def:enlace-separe-traj}
	Let $\Gamma$, $\Delta$ be trajectories {of $\xi$}. 
	\begin{enumerate}
		\item We say that $\Gamma$ and $\Delta$ are {\em interlaced} if there is a common adapted chart $C$ such that $\gamma_C$ and $\delta_C$ are {interlaced}.
		\item We say that $F$ is a \emph{projection adapted to $(\Gamma,\Delta)$}  if $F:U\subset\mathbb R^3\to\mathbb R^2$ is a definable submersion on an open domain $U$ with $\lim_{p\to 0}F(p)=0$, and $\Gamma\cap U$ (resp. $\Delta\cap U$) is a representative of the germ at $0$ of $\Gamma$ (resp. $\Delta$).
	\end{enumerate}
\end{definition}

By Proposition \ref{prop:interlacement}, the definition of interlacement does not depend on the chart $C$ if one of the trajectories has the property of regular separation and if they have flat contact. 

{In the following theorem, we provide a stronger version of the alternative  ``{interlaced} versus {separated}'' proven in  \cite[Th\'eor\`eme I]{Can-Mou-San2} for non-oscillating trajectories of analytic vector fields which  share the same iterated tangents  (see Section \ref{sec:integral-pencils} for details). More precisely, our version replaces an existential quantifier by a universal one, and is valid for vector fields definable in a polynomially bounded o-minimal expansion of $\mathbb R$. }


\begin{theorem}\label{th:defvf}
	Let $\Gamma, \Delta$ be trajectories of $\xi$ such that $\Gamma$ has regular separation property and $\Gamma, \Delta$ have flat contact. Then, either $\Gamma$ and $\Delta$ are interlaced, or for {any} definable projection $F$ adapted to $(\Gamma,\Delta)$, the germ at $0$ of $F(\Gamma)\cap F(\Delta)$ has a connected representative. {The two properties of this alternative are mutually exclusive.}
\end{theorem}

\begin{proof}
	We suppose that $\Gamma$ and $\Delta$ are not interlaced. Choose a chart $C=(V,X=(x_1,x_2,x_3))$ adapted to both $\Gamma$ and $\Delta$ and fix $F=(F_1,F_2):W\to \mathbb R^2$ a projection adapted to $(\Gamma, \Delta)$. We shall conclude that the germ at $0$ of $F(\Gamma)\cap F(\Delta)$ has a connected representative.
	
	First, suppose that $dF_1\wedge dF_2(X^{-1}(x,\gamma_C(x)))$ vanishes identically for small $x$.
	Since $F$ is a submersion, we get that $F_1$ and $F_2$ are both constant over $\Gamma$, and since $X^{-1}(x,\gamma_C(x))$ {tends to} $0$ when $x$ does,
	$F(x,\gamma_C(x) )=0$ for small $x$. Now $x\mapsto (F_1^2+F_2^2)(x,\delta_C(x))$ belongs to the Hardy field $\mathcal F(x,\gamma_C,\delta_C)$, then either vanishes identically or else does not vanish for small $x$. Depending whether it vanishes or not, $\{0\}$ or $\emptyset$ is a connected representative of the germ at $0$ of
	$F(\Gamma)\cap F(\Delta)$.
	
	Suppose now that $dF_1\wedge dF_2(X^{-1}(x,\gamma_C(x)))$ does not identically vanish for small $x$. From regular separation of $\Gamma$, ultimately, it never vanishes. Up to {replacing} $F_1$ by $-F_1$, we suppose that $dF_1(\xi)>0$ over $\Gamma$.
	For $i=1,2,3$, define
	\[ V_i:=\{p\in W\cap V;\; dF_1(p)\wedge dF_2(p)\wedge dx_i(p) \neq 0\}.\]  For a given $p$, since $(x_1,x_2,x_3)$ is a diffeomorphism, $(dx_1(p),dx_2(p),dx_3(p))$ is a basis of the dual of $T_p\mathbb{R}^3$. So at least one form among them, say $dx_k(p)$, is independent from $(dF_1(p),dF_2(p))$. Since $F$ is a submersion, $(dF_1(p),dF_2(p))$ has rank $2$ so $dF_1(p)\wedge dF_2(p)\wedge dx_k(p)\neq 0$. This shows that $V_1\cup V_2\cup V_3 =  W\cap V$.\\
	These sets $V_i$, $i=1,2,3$, are open and definable. By regular separation of $\Gamma$, there is an index $k$ such that $(x,\gamma_C(x))$ belongs to $V_k$ for small $x$. So $(F_1,F_2,x_k)$ is a local diffeomorphism on $V_k$, and $V_k$ contains a representative of the germ of $\Gamma$ at $0$. Up to shrinking $V_k$, we assume that $(F_1,F_2,x_k)$ is injective and $\Gamma$ has a connected representative contained in $V_k$.
	So $C' = (V_k,(F_1,F_2,x_k))$ is a chart adapted to $\Gamma$.
	From Proposition \ref{prop:flat}, it is also adapted to $\Delta$ up to shrinking again $V_k$. From Proposition \ref{prop:interlacement} and Theorem \ref{th:main}, $F_2(t,\gamma_{C'}(t))-F_2(t,\delta_{C'}(t))$ belongs to a Hardy field. Then either it vanishes identically, and the germ at $0$ of $F(\Gamma)\cap F(\Delta)$ coincides with the germ at $0$ of $F(\Gamma)$, which has a connected representative, or ultimately it does not vanish, and $\emptyset$ is a connected representative of  $F(\Gamma)\cap F(\Delta)$.
	
	Finally, to prove that {the properties} are mutually exclusive, {assume that $\Ga,\Delta$ are interlaced and take $C=(V,(x,y_1,y_2))$ a chart adapted to $\Ga$ and $\Delta$ such that $\gamma_C,\delta_C$ are interlaced.} Then $\{(x,y_1)\,:\, x>0,\, y_1=\gamma_{C,1}(x)=\delta_{C,1}(x)\}$ is a discrete infinite sequence in $\R^2$ that approaches $(0,0)$. The map $F:V\rightarrow \R^2,\ F(x,y_1,y_2)=(x,y_1)$ is a definable projection adapted to $\Ga,\Delta$ such that the germ at 0 of $F(\Gamma)\cap F(\Delta)$ has no connected representative.\end{proof}


\section{Integral pencils of analytic vector fields.}\label{sec:integral-pencils}

Let $\xi$ be a real analytic vector field in a neighborhood of $0\in\R^3$ such that $\xi(0)=0$. Such a $\xi$ (up to restricting its domain) is definable in the o-minimal and polynomially bounded structure $\R_{an}$ of globally subanalytic sets \cite{vdD}.  Our purpose is to apply the results of the previous section to trajectories of $\xi$ asymptotic to a formal curve, and to clarify their  {links with} F. Cano et al.\,{works}  \cite{Can-Mou-San1,Can-Mou-San2}. 
{We first recall some notions of these papers.}

{A \emph{(real irreducible) formal curve} $\mathcal{C}$ at $(\R^3,0)$ is an equivalence class  of formal parameterizations $\mathcal{C}(t)\in(t\R[[t]])^3\setminus\{(0,0,0)\}$ (up to 
	{formal} reparameterization). One can bi-univocally associate to a formal curve its sequence $IT(\mathcal{C})=\{p_n\}_{n\ge 0}$ of {\em iterated tangents} (also called sequence of {\em infinitely near points}, see for instance \cite{Wal,Cas}).}
{The sequence $IT(\mathcal{C})$ is obtained recursively as follows: we set $p_0:=0$, $\mathcal{C}_0:=\mathcal{C}$, and for $j\ge 0$, if $\pi_j$ is the blow-up centered at $p_j$, $\mathcal C_{j+1}$ is the strict transform
	of $\mathcal{C}_j$ by $\pi_j$ and $p_{j+1}:=\mathcal C_{j+1}(0)$.}
{Replacing projective blow-ups by spherical ones (see \cite{dumortier:singul-vect-field-plane} or \cite{martin-rolin-sanz:local-monom-gener-funct}) in this construction, $\mathcal C$ provides two sequences $\mathcal{C}^+=\{p_n^+\}$, $\mathcal{C}^-=\{p_n^-\}$ of {\em oriented iterated tangents} for $\mathcal{C}$, each one {corresponding to} a (formal) {\em half-branch} of $\mathcal{C}$ {(determined by the sign of $t$)}.}
Given a half branch $\mathcal{C}^\epsilon$, $\epsilon \in \{+,-\}$, and a subanalytic set $A$ with $0\in\overline{A}$, we say that $\mathcal{C}^\epsilon$ {\em is contained in $A$} if 
{for all $j\ge 0$, $p^\epsilon_{j+1}\in A_{j+1}:=\overline{\rho_j^{-1}(A_{j}\setminus\{p_{j}^{\varepsilon}\})}$ where $\rho_j$ is the spherical blow-up at $p_{j}^{\varepsilon}$ and $A_0=A$.}
If $\mathcal{C}(t)$ is convergent,
{the set $C_{\delta}=\{\mathcal C(t);\; t\in (0,\delta)\}$ materializes one half branch $\mathcal C^{\epsilon}$, and $\mathcal C^{\epsilon}$ is contained in $A$}
{means that $C_{\delta}\subset A$ for small $\delta>0$}.

Let $\Ga$ be a trajectory {at $0$} of the vector field $\xi$. Following \cite[p. 287]{Can-Mou-San1}, we say that $\Ga$ has the {\em property of iterated (oriented) tangents} if we can associate a sequence $IT(\Ga)=\{q_n\}_{n\ge 0}$ to $\Ga$ by the {following process: set $q_0:=0$, $\Ga_0:=\Ga$ and for $j\ge 0$, the point $q_{j+1}$ is the unique accumulation point  
	of $\Ga_{j+1}:=\rho_{j}^{-1}(\Ga_{j})$ in $\rho^{-1}_{j}(q_{j})$, where $\rho_j$ is the spherical blow-up at $q_{j}$.}
{F}or instance, if $\Ga$ is \emph{non-oscillating} with respect to semi-analytic sets {(i.e. its intersection with any semi-analytic set has finitely many connected components)}, {\cite[Proposition 1.2]{Can-Mou-San1} shows that $\Ga$ has iterated tangents.
	
	\begin{definition}
		We say that a trajectory $\Ga$ is {\em asymptotic to a half-branch} $\mathcal{C}^\epsilon$ of a formal curve at $(\R^3,0)$ if $\Ga$ has the property of iterated tangents and $IT(\Ga)=\mathcal{C}^\epsilon$. The set $\mathcal{P}_{\mathcal{C}^\epsilon}$ composed of all trajectories of $\xi$ that are asymptotic to a given  half-branch $\mathcal{C}^\epsilon$ is called the \emph{integral pencil {of $\xi$} with (half)-axis} $\mathcal{C}^\epsilon$.
	\end{definition}
	A trajectory can be asymptotic to at most one half-branch $\mathcal{C}^\epsilon$ and thus different formal half-branches determine disjoint integral pencils.
	On the other hand, given  a formal curve $\mathcal{C}$, if $\mathcal{P}_{\mathcal{C}^\epsilon}\ne\emptyset$ for one of its half-branches $\mathcal{C}^\epsilon$, then $\mathcal{C}$ is {\em invariant} for the vector field $\xi$, i.e., being $\mathcal{C}(t)$ a parameterization, one has $\xi|_{\mathcal{C}(t)}=h(t)\frac{d\mathcal{C}(t)}{dt}$ for some $h(t)\in\R[[t]]$ ({the proof given in \cite[Prop. 2.1.1]{Can-Mou-San1} in the convergent case applies {to} formal curves}).
	{If} $\mathcal{C}$ is not contained in the singular locus of $\xi$, we say that the axis is {\em non-degenerated}. 
	
	\begin{lemma}\label{lm:flat-contact-pencil}
		Let $\mc P_{\mathcal{C}^+}$ be a pencil with a formal non-degenerated axis and  let $\Ga,\Delta$ be two trajectories in $\mc P_{\mathcal{C}^+}$. Then  $\Ga$ and $\Delta$ have flat contact.	
	\end{lemma}
	\begin{proof}
		Let $\mathcal{C}$ be the formal curve at $(\R^3,0)$ {having $\mathcal{C}^+$ as} one of its half-branches.  Let $B=(U,(x,y,z))$ be an analytic chart at $0\in\R^3$ such that $\mathcal{C}^+$ is not contained in $x=0$. Since $\mathcal{C}$ is non-degenerated, $\mathcal{C}$ is not contained in the set $\{dx(\xi)=0\}$. Up to considering a finite composition of blow-ups at points in the sequence of iterated tangents of $\mathcal{C}^+$, we can assume that $dx(\xi)$ has constant nonzero sign over an open subanalytic set containing the germs of $\Ga$ and $\Delta$.  
		This shows that the chart $B$ is adapted to $\Ga$ and $\Delta$, up to replacing $x$ by $-x$.
		
		Consider a parameterization of $\mathcal{C}$ of Puiseux type
		$\mathcal{C}(t)=(t^\nu,\theta(t))\in(t\R[[t]])^3$ in the chart $B$. From the definition of iterated tangents, we get the following property for $\g_B$: 	%
		\begin{equation}\label{eq:asymptotic}
			\forall N\in\N_{\ge 1},\ \; \|\g_B(t^\nu)-J_N\theta(t)\|=o(t^N),\;\mbox{ for }t>0,
		\end{equation}
		where $J_N\theta(t)$ denotes the truncation of $\theta(t)$ up to order ~$N$. We say that $\gamma_B$ has $\mathcal{C}$ as \emph{Puiseux expansion} if (\ref{eq:asymptotic}) holds.
		
		Similarly, $\delta_B$ has $\mathcal C$ as Puiseux expansion. Thus, given $n\in\N$, if $\varepsilon_n>0$ is such that
		\[ 
		\sup\{\|\g_B(t^\nu)-J_{n\nu}\theta(t)\|,\|\g_B(t^\nu)-J_{n\nu}\theta(t)\|\}\le \frac{1}{2}t^{n\nu},\;\mbox{ for }0<t<\varepsilon_n,
		\]
		then we obtain that $\|\g_B(x)-\delta_B(x)\|\le x^n$ for any $x<\varepsilon_n^{\nu}$. 
		The associated respective parameterizations $\g_B,\delta_B:(0,a)\to\R^2$ of $\Ga,\Delta$ have flat contact as required.
	\end{proof}
	\begin{remark}
		We notice that {for a trajectory, to have a given Puiseux expansion is necessary and sufficient to belong to the corresponding integral pencil. More precisely}, if $\Ga$ is a trajectory, $B=(U,(x,y,z))$ an adapted chart, $(x,\g_B(x))$ the associated parameterization and $\mathcal{C}^+$ the half-branch of $\mathcal C(t)=(t^\nu,\theta(t))$ contained in $x>0$, then $\Ga$ belongs to $\mc P_{\mathcal{C}^+}$ if and only if $\g_B$ satisfies (\ref{eq:asymptotic}).
		
		
	\end{remark}
	
	\strut
	
	Let us prove that an integral pencil with a non-degenerated axis has at least {one} trajectory with the regular separation property.
	
	\begin{theorem}\label{pro:formal-axis}
		Let $\mc P_{\mathcal{C}^+}$ be an integral pencil with a non-degenerated axis $\mc C^+$ of an analytic vector field $\xi$ at $(\R^3,0)$. Let $S$ be a subanalytic set such that $\mathcal{C}^+\subset S$ {(in the sense of the second paragraph of the present section)} and having minimal dimension among those subanalytic sets with this property. Then, the following holds:
		\begin{enumerate}
			\item There exists at least one trajectory in the pencil $\mc P_{\mathcal{C}^+}$ contained in $S$.
			\item {A trajectory $\Ga\in\mc P_{\mathcal{C}^+}$}
			has the regular separation property if and only if $\Ga\subset S$ {in a neighborhood of $0$}.
		\end{enumerate}
	\end{theorem}
	
	\begin{proof}
		Let $s=\dim S$.
		If $s=1$ then $\mathcal{C}$ is a convergent analytic curve and $S$  contains the connected component, $\Ga$, of $\mathcal{C}\setminus\{0\}$ whose oriented iterated tangents correspond to $\mathcal{C}^+$. Since $\Ga$ is invariant for $\xi$ and not contained in the singular locus of $\xi$, $\Ga$ satisfies (i). The conclusion of (ii) is also clear in this case: a trajectory having the regular separation property and flat contact with the trajectory $\Ga$ (definable) must coincide with $\Ga$ (as germs).
		
		Assume that $s>1$.
		Up to taking a stratification of $S$, we may assume that $S$ is an analytic submanifold of pure dimension $s$. Moreover, we may also assume that $S$ is everywhere tangent to $\xi$ (otherwise, the locus of tangency between $S$ and $\xi$ {would} be a subanalytic set containing $\mathcal{C}^+$ and of dimension less than $s$).
		
		In order to prove (i), it suffices to prove it after a blow-up $\pi_Y:M\to\R^3$ with smooth analytic center $Y$ through $0$ such that $\dim(Y)\le 1$. 
		In fact, for any such center we have $\mathcal{C}^+\not\subset Y$ 
		(since $\mathcal{C}^+$ is not convergent because $s>1$), 
		so that the transform $\wt{\mathcal{C}^+}$ 
		of $\mathcal{C}^+$ by $\pi_Y$ is a well defined formal half-curve at some $p\in M$. 
		Moreover, $\wt{\mathcal{C}^+}$ is invariant by a vector field $\wt{\xi}$ at $(M,p)$ and, 
		using the characterization (\ref{eq:asymptotic}) {via Puiseux expansions}, 
		$\pi_Y$ establishes a bijection between the pencils $\mc P_{\wt{\mathcal{C}^+}}$ and 
		$\mc P_{\mathcal{C}^+}$ of $\wt{\xi}$ and $\xi$ respectively. With this remark and using 
		Hironaka's Rectilinearization Theorem \cite{Hir}, we may assume that $S$ is a smooth 
		analytic manifold at $0$. 
		When $s=2$, item (i) is a consequence of Seidenberg's reduction of planar vector fields 
		\cite{Sei} ($\mathcal C$ correspond to either a stable, unstable, or central manifold, up to further punctual blowing-ups). When $s=3$, item (i) follows from Bonckaert \cite[Theorem 2.1]{bon}.

		Let us prove (ii) in the case $s>1$. Suppose {that} $\Ga\in \mc P_{\mathcal{C}^+}$ is contained in $S$ near $0$.
		Consider an analytic chart $B=(U,(x,y_1,y_2))$ at the origin in which $\mathcal{C}$ has a Puiseux's parameterization $\mathcal{C}(t)=(t^\nu,\theta(t))$ and so that $\mathcal{C}$ is transversal to $\{x=0\}$. Up to changing the sign of $x$, we may assume that $\Ga\subset\{x>0\}$. Note that $\Ga$ is {non-oscillating with respect to semi-analytic sets} {by \cite[Th\'eor\`eme 1]{Can-Mou-San1}} since $\mathcal{C}$ is divergent ({if it were convergent, we would have $\Gamma$ equal to one of its half branches and in fact $s=1$}). Thus, according to {\cite[Lemme 1.9]{Can-Mou-San2}}, the chart $B$ is adapted to $\Ga$.	
		Consider the solution $\g_B:(0,a)\to\R^2$ associated to $\Ga$ in this chart and let us show that $\g_B$ satisfies the two conditions (i) and (ii) of Definition \ref{def:separatedtrajectories}.
		
		Let $h$ be a subanalytic function such that $\Ga\subset {\rm Dom}(h)$. We notice that $\Ga$ is also {non-oscillating} with respect to subanalytic sets (\cite[Cor. 1.5]{Can-Mou-San2}) so that the function $x\mapsto h(x,\g_B(x))$ has ultimately a sign when $x\to 0$. Assume that it is positive for any $x\in (0,a)$ and that $\lim_{x\to 0}h(x,\g_B(x))=0$.
		
		First, by (\ref{eq:asymptotic}), we have that $\|\g_B(x)\|\leq x^{1/N}$ for some $N>0$ sufficiently big and for small $x$. Using Lion's Preparation Theorem for subanalytic functions \cite{Lio} and again the {non-oscillation} of $\Ga$ with respect to subanalytic sets, we obtain that  $|h(\g_B(x))|\le x^{1/N'}$ for $N'>0$ sufficiently big. This proves the condition (ii) of Definition \ref{def:separatedtrajectories}. 
		
		{Let us now show condition (i) of Definition \ref{def:separatedtrajectories}. 
			For this, assume {that} $h_{|\Gamma}$ does not identically vanish near $0$.}
		Taking a subanalytic stratification adapted to $S$ and $Z:=h^{-1}(0)$, we may assume that $S\cap Z=\emptyset$. 
		{Let} 
		$\mathcal{C}^+:=\{p_n\}$, 
		$M_0:=\R^3$, and {let}
		$\pi_n:M_n\to M_{n-1}$ {be}  the spherical blow-up at $p_{n-1}$ for $n\ge 1$. Denote also 
		{$\Pi_n:=\pi_{1}\circ\pi_{2}\circ\cdots\circ\pi_n:M_n\to\R^3$.}
		By minimality of $s$,
		$\mathcal{C}^+\not\subset\partial S$, and hence there exists $n\ge 1$ 
		such that $p_n\not\in\overline{\Pi_n^{-1}(\partial S\setminus\{0\})}$. 
		Let $V$ be an open neighborhood of $p_n$ in $M_n$ such that 
		$V\cap\Pi_n^{-1}(\partial S\setminus\{0\})=\emptyset$. 
		{Set} $g:=h\circ\Pi_n$ and let us show that
		\begin{equation}\label{eq:zeros-g}
			\{g=0\}\,\cap V\cap\,\overline{\Pi_n^{-1}(S\setminus\{0\})}\subset D:=\Pi_n^{-1}(0).
		\end{equation}
		Let $q\in V\cap\overline{\Pi_n^{-1}(S\setminus\{0\})}$ such that $q\not\in D$ and suppose that $g(q)=0$. Then 
		$\Pi_n(q)\in\overline{S}\setminus\{0\}$ and $\Pi_n(q)\in Z$. 
		Since $\overline{S}\cap Z\subset\partial S$, we would have 
		$q\in\Pi_n^{-1}(\partial S\setminus\{0\})$, which is impossible because 
		$\Pi_n^{-1}(\partial S\setminus\{0\})\cap V=\emptyset$. This shows (\ref{eq:zeros-g}).
		
		{Now, the map $f:={x\circ\Pi_n}_{|V}$ is} a subanalytic function in $V$ satisfying 
		\[ \{f=0\}\cap\overline{\Pi_n^{-1}(S\setminus\{0\})}\subset D\cap V.\] Applying \L{}ojasiewicz Inequality to $f,g$ (see for instance \cite{Hir} or \cite{Bie-M}), there exists $c>0$, $\alpha>0$ such that
			\[ 
		|g(q)|\ge c |f(q)|^\alpha,\;\forall q\in\overline{\Pi_n^{-1}(S\setminus\{0\})}.\]
		Using this last inequality for the points $q\in\Pi_n^{-1}(\Ga)\subset\Pi_n^{-1}(S\setminus\{0\})$ we {finally} obtain for any $x>0$ sufficiently small {that}
		\[
		h(x,\g_B(x))\ge c (x(x,\g_B(x)))^\alpha=c x^\alpha.
		\]
		This proves the regular separation property for the solution $\g_B$ (condition (i) of Definition \ref{def:interlaced-separated}) and finishes the proof of the ``if'' part of  Theorem \ref{pro:formal-axis} (ii).
		
		Finally, let us prove the ``only if'' part of Theorem \ref{pro:formal-axis} (ii). 
		Let $\Delta\in P_{\mathcal{C}^+}$. 
		{Assume }that $\Delta$ has the regular separation property.
		Consider $B=(V, (x,y_1,y_2))$ an adapted chart for $\Delta$ and denote 
		$x\mapsto (x,\delta_B(x))\in V$
		the parameterization of $\Delta$ in this chart. 
		Since $\delta_B$ has the regular separation property, either $(x,\delta_B(x))\in S$ ultimately, 
		or $(x,\delta_B(x))\notin S$ for all small $x$ (Lemma \ref{lem:1}). We assume the latter,
		and get to a contradiction. 
		From {the} already proven {assertion} (i) and {the} ``if'' part of {assertion} (ii) in Theorem \ref{pro:formal-axis},
		$S$ contains a trajectory $\Gamma$ with regular separation.
		By Lemma \ref{lm:flat-contact-pencil}, $\Gamma$ and $\Delta$ have flat contact. 
		Then from Proposition \ref{prop:flat}, $B$ is adapted to $\Gamma$, 
		and $\gamma_B$ and $\delta_B$ have flat contact,
		where $(x,\gamma_B(x))$ parameterizes $\Gamma$. 
		Let $d:V\to\mathbb R$ be the subanalytic function $p\mapsto \inf_{q\in S\cap V}||p-q||$, that is, the distance to ${S}$ (in the chart $B$).
		By regular separation, $d(x,\delta_B(x))$ is either identically zero or bounded {from} below by a power
		of $x$. But $d(x,\delta_B(x))\le ||\delta_B(x)-\gamma_B(x)||$ and $||\delta_B-\gamma_B||$
		is ultimately smaller than any power of $x$, so $d(x,\delta_B(x))\equiv 0$, which means {that}
		$(x,\delta_B(x))\in\overline{S}$. Finally, $\Delta\subset \overline{S}\setminus S=\partial S$ near $0$.
		Since $IT(\Delta)=\mathcal C^+$, this implies {that} $\mathcal C^+\subset \partial S$. 
		But $\dim \partial S < \dim S$, which contradicts the minimality of $s$.
		This achieves the proof of Theorem \ref{pro:formal-axis}.\end{proof}

	Using Theorem~\ref{pro:formal-axis} and Lemma~\ref{lm:flat-contact-pencil}, we can apply the results obtained in section~\ref{sec:trajectories} to pairs of trajectories in a pencil with {non-degenerated} axis $\mathcal{C}^+$, provided that {one of them} belongs to {the} subanalytic set $S$ of minimal dimension {$s$} which contains the formal half-axis.
	
	\vspace{.2cm}
	
	{\bf Application to pencils with analytic axis.-}
	The case 
	$s=1$ corresponds to 
	{convergent axis.}
	{Then the unique trajectory $C^+$ of $\mc P_{\mathcal C^+}$ having the regular separation property
		is obtained as the sum of the series $\mathcal C(t)$ with a sign condition on $t$.}
	In this case, Theorem~\ref{th:main} 
	{says that, given another}
	trajectory $\Delta\in\mc P_{\mc C^+}$, either the pair $C^+,\Delta$ is interlaced or 
	the ring of germs of subanalytic functions restricted to $\Delta$ is a Hardy field.
	It is well known that the latter is equivalent to being non-oscillating {with respect to analytic sets}, (and non-oscillation with respect to
	analytic sets or to subanalytic sets {are} equivalent for trajectories, see \cite[Corollaire 1.5]{Can-Mou-San2}), so
	in this case, Theorem~\ref{th:main} only recovers the alternative non-oscillating/spiraling of \cite[Th\'eor\`eme 1]{Can-Mou-San1}.
	
	Despite this discussion, Theorem~\ref{th:defvf} {provides us with} a little bit {more information} in the non-oscillating situation: since the axis $ C^+$ 
	has also the regular separation property {in}
	any polynomially bounded o-minimal structure {wich} defines the subanalytic sets, 
	we can enlarge the class of sets 
	a trajectory in the pencil 
	{does not oscillate with.}
	More precisely, we have:

	\begin{corollary}\label{cor:analytic-axis}
		Let $\mc P_{\mathcal{C}^+}$ be an integral pencil of an analytic vector field at $(\mathbb{R}^3,0)$ with a non-degenerated convergent axis {$\mathcal{C}^+$}. Let $\Delta\in\mc P_{\mathcal{C}^+}$ be a non-oscillating trajectory of the pencil parameterized by $(x,\delta_B(x))$ in an analytic chart $B=(U,(x,y_1,y_2))$ at the origin for which $\mc C^+\subset\{x>0\}$. Then, for any polynomially bounded o-minimal structure $\mathcal{R}$ {expanding} $\mathbb{R}_{an}$, the ring $\mathcal{F}_{\mathcal{R}}(x,\delta_B)$ of germs at $x=0$ of functions of the form $x\mapsto h(x,\delta_B(x))$, where $h$ is definable {in} $\mathcal{R}$, is a Hardy field.
	\end{corollary}

	\vspace{.2cm}
	
	{\bf Applications to transcendental axis.-}
	If $\mathcal C^+$ is not included in {any} subanalytic set of positive codimension, (this is, when $s=3$), we will say that the axis $\mathcal{C}^+$ is {\em subanalytically transcendental}. In this case, all the trajectories in the pencil have the regular separation property and the results in section~\ref{sec:trajectories}  fully apply to any pair of trajectories {in it}.

	We remark that interlaced integral pencils of non-oscillating trajectories provide examples of pencils with subanalytically transcendental formal axis (see \cite[Th\'{e}or\`{e}me II and discussion p. 29]{Can-Mou-San2}).
	In {this} case, Theorem~\ref{th:main} gives no more information for the relative behavior of a pair of trajectories than the one already established in that reference: the trajectories are two by two interlaced.
	
	On the {other hand}, if we have a non-interlaced integral pencil with subanalytically transcendental axis, as already mentioned in section~\ref{sec:trajectories}, we obtain a better description than the ``s\'{e}par\'{e}'' condition of \cite{Can-Mou-San2} for {pairs} of trajectories. {The following results summarizes this new contribution}:

	\begin{corollary}\label{cor:separated-trans-pencil}
		Let $\mc P_{\mathcal{C}^+}$ be a separated (i.e. non-interlaced) integral pencil of an analytic vector field at $(\mathbb{R}^3,0)$ with a formal subanalytically transcendental axis. Then, given any pair of trajectories $\Ga,\Delta\in\mc P_{\mc C^+}$, the following assertions hold:
		\begin{enumerate}
			\item Any analytic chart $B=(U,(x,y_1,y_2))$ at the origin such that $\mathcal{C}^+\subset\{x>0\}$ is adapted to both $\Ga,\Delta$ and, being $(x,\g_B(x))$ and $(x,\delta_B(x))$ their respective parameterizations in $B$, the ring $\mathcal{F}_B(x,\g_B,\delta_B)$ of germs at $x=0$ of functions of the form $x\mapsto h(x,\g_B(x),\delta_B(x))$, where $h$ is a subanalytic function, is a Hardy field.
			\item For any subanalytic submersion $F:V\to\R^2$ defined in an open set $V\subset\R^3$ containing $\Ga\cup\Delta$, and satisfying $\lim_{p\to 0}F(p)=0$, either $F(\wt{\Ga})=F(\wt{\Delta})$ or $F(\wt{\Ga})\cap F(\wt{\Delta})=\emptyset$ for some representatives $\wt{\Ga}$, $\wt{\Delta}$ of $\Ga,\Delta$, respectively.
		\end{enumerate}
	\end{corollary}

	\section{Transcendental axis and the SAT condition - Examples.}\label{sec:examples}
	
	In order to {apply Corollary~\ref{cor:separated-trans-pencil}} 
	we need to construct separated pencils with subanalytically transcendental axis. So far, we do not know any such example in the literature. This section is devoted to provide examples of such integral pencils, with different ``sizes'': the pencil may consist of a solitary trajectory, a ``surface'' of trajectories (one-parameter family) or a whole ``open domain'' of trajectories (a two-parameter family).

	The main ingredient used to produce examples is the so called {\em Strongly Analytic Transcendence (SAT) property} introduced in {J.-P. Rolin, R. Schaefke and the third author's} work \cite{Rol-San-Sch}. By its very definition, it expresses a condition of transcendence of a formal curve with respect to analytic sets, {strengthened} by the possibility {of right composition by polynomials} (we recall the definition below). We prove that
	the SAT property assures the transcendence with respect to subanalytic sets (Proposition~\ref{pro:SAT-pencil}). For that, we show that SAT is preserved by local blow-ups with analytic smooth centers and use Hironaka's Rectilinearization Theorem \cite{Hir}.
	
	\strut
	
	A first difficulty is that an analytically transcendental axis (i.e., a formal invariant curve not contained in {any} {analytic} 
	set of positive codimension) may not be subanalytically transcendental, as the following example shows.

	\begin{example}[see also a similar example in \cite{LG:solitary}]\label{ex:an-tr-not-san-tr}\rm
		Let $\xi$ be the analytic vector field 
		\[ \xi(x,y,z) = x^3\frac{\partial}{\partial x} + x(y-x)\frac{\partial}{\partial y} + z(y-x)(1-x)\frac{\partial}{\partial z}.\]
		The vector field $\xi$ is singular at $0$ and $(x, y(x), z(x))$ parameterizes a trajectory of $\xi$ for $x>0$ if and only if
		the functions $y,z$ satisfy the {following} system of differential {equations}:
				\[ 
		\left\{ 
		\begin{array}{lclr}
			 x^2 y' & = & y-x & (a)
			 \\ 
			 x^3z' & = &  z(y-x)(1-x) & (b)
		 \end{array} 		 \right. 	 \]
		The differential equation $(a)$ does not depend on $z$. It is the classical Euler differential equation, whose solutions for $x>0$ all {tend to} $0$ and are asymptotic to the formal {power} series 
		\[ \hat{h}(x) := \sum_{n=0}^\infty n! x^{n+1}.\]
		Now, {by} dividing (b) by $x^3$ and since $y-x=x^2y'$, we get {that}
		\[ z'(x) = z(x)\left(\frac{1}{x}+\left(\frac{y(x)}{x}\right)'\right),\]
		which gives $z(x)=A xe^{\frac{y(x)}{x}}$ for some $A\in\mathbb R$.
		
		{In particular, if $A=1$, then}, for any choice $h(x),\; x>0,$ of a solution of $(a)$,
		the image of $x\mapsto (x, h(x), xe^{\frac{h(x)}{x}})$ is a trajectory of $\xi$ at $0$, 
		{which} is asymptotic to the axis $\mathcal C^+$ associated {to} the parameterized formal curve 
		\[ \displaystyle \mathcal C(t) := \left(t, \hat{h}(t), te^{{\hat{h}(t)}/{t}}\right),\]
		and contained in $\{x>0\}$.
		(The series $te^{{\hat{h}(t)}/{t}}$ is obtained from $e^{1+t}=e+et+e\frac{t^2}{2}+\dots$, as $\hat h(t)/t \in 1+t\mathbb R[[t]]$.) So the pencil $\mathcal P_{\mathcal C^+}$ contains a ({nonempty}) family of trajectories, parameterized by the choice of a solution of (a).
		
		Since $\hat{h}$ is divergent, the minimal dimension $s$ of a subanalytic set that contains $\mathcal C^+$
		is at least $2$. {Moreover,} since the subanalytic surface \[ S=\{(x,y,z)\in\mathbb R^3;\; z=xe^{\frac{y}{x}},\; 0<y<3x\}\]
		contains the trajectories of $\mathcal P_{\mathcal C^+}$, $\mathcal C^+\subset S$, {we have that} $s=2$. In particular,
		$\mathcal C^+$ is not subanalytically transcendental.
		
		It {however} happens that $\mathcal C^+$ is analytically transcendental. Indeed, if $X$ is analytic and contains $\mathcal C^+$, the intersection $S\cap X$ is a subanalytic set that contains $\mathcal C^+$, so  $\dim(S\cap X)\ge 2$, therefore $S\subset X$. But $S$ is not included in any analytic set but $\mathbb R^3$ near $0$ (this is a variation of a classical example due to Osgood, see \cite[Example 2.14]{Bie-M}). Indeed, if $f$ is analytic at $0$ and vanishes on $S$, 
		then $f(x,xy,xe^{y})$ vanishes on an open set {close} to $0$, then everywhere. Assuming {that} $f(x,y,z)=\sum_{(i,j,k)\in\mathbb N^3} f_{ijk}x^iy^jz^k$,
		this can be rewritten {as}  \[ f(x,xy,xe^{y}) = \sum_{n\in\mathbb N}x^n\sum_{i+j+k=n} f_{ijk}\,y^j\,(e^y)^k =0,\] which implies {that} all $f_{ijk}$ are zero {since} $e^y$ is transcendental over $\mathbb R[y]$.
	\end{example}

	In the {previous} example, the formal curve $\mathcal{C}$ is analytically transcendental but its transform {by}
	the blow-up at the origin is not analytically transcendental anymore. The following lemma clarifies the relation between subanalytic and analytic transcendence in terms of blow-ups. {We leave the proof to the reader}. It is a consequence of Hironaka's Rectilinearization Theorem \cite{Hir} which asserts that any subanalytic set can be transformed into a semianalytic set by a finite sequence of local blow-ups with smooth analytic centers.

	\begin{lemma}\label{lm:transcendence}
		Let $\mathcal{C}$ be a formal curve at $0\in\R^n$. The following are equivalent:
		\begin{enumerate}
			\item The curve $\mathcal{C}$ is subanalytically transcendental.
			\item The curve $\mathcal{C}$ is analytically transcendental and its strict transform by any finite sequence of local blow-ups with smooth analytic centers is also analytically transcendental.
		\end{enumerate}
	\end{lemma}

	We now recall the SAT property, introduced in \cite{Rol-San-Sch} for formal solutions of systems of ordinary differential equations. We {first} need the following {notions}. For a given $q\ge 1$,  
	a polynomial $P(x)\in\R[x]$ {is}
	called {\em $q$-short} if $P(0)=0$ and $\text{deg}\,P<(q+1)\text{val}\,P$, where $\text{val}\,P$ {denotes the order of $P$ at $0$}.
	We will say that $P$ 
	is {\em positive} if $P^{(\text{val}\,P)}(0)>0$ (where $P^{(k)}$ is the $k$-th derivative of $P$). Also, as a matter of notation, {for} $k\ge 0$ and $h(x)\in\R[[x]]$ is a formal power series, we {write}
	\[ 
	T_kh(x)=\frac{h(x)-J_kh(x)}{x^k},
	\]
	where $J_kh(x)$ is the truncation of $h(x)$ up to and including degree $k$ (thus, in particular, $T_kh(0)=0$).
	\begin{definition}\label{def:sat}
		Fix $q\in\N_{\ge 1}$. Let $H(x)=(H_1(x),...,H_n(x))\in({x}\R[[x]])^n$.
		We say that $H(x)$ is  {\em $q$-strongly analytically transcendent} ($q$-SAT 
		for short, or 
		SAT if $q$ is {given}) if for any $k\in\mathbb N{, l}\in\mathbb N_{\ge 1}$  and for any family of distinct $q$-short positive polynomials $P_1,...,P_l$ we have {that}
		\begin{equation}\label{eq:sat}
			f \Bigl(
			x,T_kH\big(P_1(x)\big),T_kH\big(P_2(x)\big),...,T_kH\big(P_l(x)\big) \Bigr)=0\Rightarrow f\equiv 0
		\end{equation}
		for any convergent series $f(x,Z)\in\R\{x,Z\}$ in $1+nl$ variables.
	\end{definition}

	\begin{remark}\label{rk:sat-under-diffeo}
		{{Condition (\ref{eq:sat})} means that the formal curve $\mathcal C^k_{H,P}(x)$ is analytically transcendental, where $H(x)=(H_1(x),...,H_n(x))$, $P=(P_1,\dots,P_l)$ and 
			\[ \mathcal C^k_{H,P}(x)= \Bigl(x,T_kH\big(P_1(x)\big),T_kH\big(P_2(x)\big),...,T_kH\big(P_l(x)\big)\Bigr).\]
			Furthermore, analytic transcendence is preserved under analytic isomorphism. Indeed, if $\mathcal C(x)\in (x\mathbb R[[x]])^m$ is 
			transcendental and $\Psi:(\mathbb R^m,0)\to(\mathbb R^m,0)$ is analytic and invertible at $0$, then $f\circ\Psi\circ \mathcal C(x)=0$ implies
			$f\circ\Psi =0$, that gives $f=0$ since $\Psi$ is invertible. 
			Together with the fact that a curve
			$\mathcal C^k_{\theta\circ H,P}(x)$ can be expressed as 
			$\mathcal C^k_{\theta\circ H,P}(x)=\Psi(\mathcal C^k_{H,P}(x))$
			with an invertible $\Psi$ if $\theta$ is invertible, 
			this shows that the $q$-SAT property is preserved under analytic isomorphism.}
		
		{The stability of analytic transcendence can be  {pushed} further as follows, and we  {will} use  {this later on}. 
			If $\mathcal C(x)\in (x\mathbb R[[x]])^m$ is transcendental and $\Psi:(\mathbb R^m,0)\to(\mathbb R^m,0)$ is analytic and  invertible in an open set $U$ with $0\in \overline{U}$, then $\Psi\circ \mathcal C(x)$ is transcendental. Indeed, if $f$ is analytic, $f\circ\Psi\circ \mathcal C(x)=0$ still gives 
			$f\circ\Psi =0$, so $f$ vanishes on the image of $\Psi$, which contains an open set close to $0$. So $f=0$ everywhere.
			{We will} call \emph{quasi-isomorphism} an analytic mapping $\Psi:(\mathbb R^m,0)\to(\mathbb R^m,0)$ of the form $\Psi(x,\ww)=(x,\psi(x,\ww))$ where $x\in\mathbb R$, $\ww\in\mathbb R^{m-1}$, and that is invertible in an open set with $0$ in its closure.}
	\end{remark}

	The following proposition shows that if $H$ is $q$-SAT,  {then} the composition of $H$ with distinct $q$-short positive polynomials is subanalytically transcendental.

	\begin{proposition}\label{pro:SAT-pencil}
		{Suppose  {that} $H(x)\in(x\R[[x]])^n$ has the $q$-SAT property} and let \\ $P_1(x),...,P_l(x)$ be distinct $q$-short positive polynomials. Let $\mathcal{C}$ be the 
		formal curve parameterized by 
		\[ 
		\mathcal{C}(x)=\Big(x,H\big(P_1(x)\big),H\big(P_2(x)\big),...,H\big(P_l(x)\big)\Big).
		\]
		Then $\mathcal{C}$ is subanalytically transcendental.
	\end{proposition}
	\begin{proof}
		By the SAT property, the curve $\CC(x)$ is analytically transcendental. 
		Using Lemma~\ref{lm:transcendence} and Remark~\ref{rk:sat-under-diffeo}, 
		it  {suffices} to show that if $\pi_Z$ is a local blow-up with a smooth analytic 
		center $Z$ at $0\in\R^{1+nl}$ and $\wt{\CC}$ is the strict transform of $\CC$ 
		by $\pi_Z$ then $\wt{\CC}$ can be parameterized  {by}
		\[ 
		\wt{\CC}(x)=\Psi\Big(x,T_kH\big(P_1(x)\big),...,T_kH\big(P_l(x)\big)\Big),
		\]
		where $k$ is a positive integer
		and $\Psi$ {is a quasi-isomorphism.}
		This is easily deduced from the three following facts.
		\begin{claim}\label{5:claim:1}
			If $\mathcal C(x)=(x,c(x))$ is a formal curve and $Z$ a smooth analytic center that does not contain the curve, 
			then the lift $\widetilde{\mathcal C}$ of the curve by the blow-up $\pi_Z$ with center 
			$Z$ has a parameterization of the form 
			\[ \widetilde{\mathcal C}(x)=\Phi\left(x, T_k\left(\psi\circ \mathcal C\left(x\right)\right)\right),\] for some $k> 0$ 
			and some quasi-isomorphisms $\Phi$ and $\Psi=(x,\psi)$.
		\end{claim}
		\begin{claim}\label{5:claim:2}
			If $\Psi=(x,\psi)$ is a quasi--isomorphism, and $\mathcal C(x)=(x,c(x))$
			is a formal curve, then the curve $(x,T_k(\psi\circ \mathcal C(x)))$ 
			is the image of the curve $(x,T_k c(x))$ by a quasi-isomorphism $\Phi$:
			\[ (x, T_k(\psi\circ\mathcal C(x)))=\Phi (x,T_k c(x)).\]
		\end{claim}
		\begin{claim}\label{5:claim:3}The curve  $(x,T_k(H\circ P(x)))$, with for short
			\[ H\circ P(x)=(H_i\circ P_j(x),\; {i=1,\dots, n,\; j=1,\dots,l}),\] is the image by a quasi-isomorphism
			$\Psi$ of the curve \[ \big(x,(T_kH_i)\circ(P_j(x)),\; {i=1,\dots, n,\; j=1,\dots,l}\big).\] With the same notation as above: 
			$(x,T_k(H\circ P(x)))=\Psi (x,(T_kH)\circ P(x)).$
		\end{claim}
		To complete the proof from the claims, it  {suffices to note} that Claim \ref{5:claim:1} 
		applies to our curve, as 
		it is analytically transcendental,  {hence} cannot be included in the analytic center $Z$.
		Applying Claim \ref{5:claim:2} then Claim \ref{5:claim:3} to the resulting curve gives exactly our  {assertion}.
		It only remains to prove the claims.
		
		\begin{proof}[Proof of claim \ref{5:claim:1}]
			The introduction of the quasi{-}isomorphism $\Psi$
			allows to make arbitrary choices of analytic coordinates as long as 
			the $x$ variable remains unchanged. 
			We  {consider} whether the curve is tangent to $Z$ or not. If not, we 
			choose a system of coordinates $(x,\yy,\zz)$, $\yy=y_1,\dots, y_m$, 
			$\zz=z_1,\dots,z_r$ (with $m+r=nl$), such that $Z=\{x=g(\zz),\yy=0\}$, and $c(x)=(c_{\yy}(x),c_{\zz}(x))$
			have order at least $2$ ($\mathcal C$ is tangent to the $x$ axis).
			
			We choose local coordinates $(\tilde x,\tilde \yy, \tilde \zz)$ 
			for the blow-up given by 
			$x=\tilde x,
			\yy=(\tilde x-g(\tilde \zz))\tilde \yy,
			\zz=\tilde \zz$
			(so the exceptional divisor has equation $\tilde x= g(\tilde\zz)$). In  {these} coordinates, the lift $\widetilde{\mathcal C}$ of $\mathcal C$ by $\pi_Z$ has parameterization
			$(x,\frac{c_{\yy}(x)}{x-g(c_{\zz}(x))},c_{\zz}(x))$. 
			The function $h(x,\zz)=\frac{1}{1-g(x\zz)/x}$ is analytic in a neighborhood of $0$, so 
			$\widetilde{\mathcal C}$  can be expressed in terms of $T_1c(x)$ as:
			\[ \widetilde{\mathcal C}(x)=(x, \widetilde{c}(x))=(x,h(x,T_1c_{\zz}(x))T_1c_{\yy}(x),xT_1c_{\zz}(x)) = \Phi(x,T_1c(x)),\]
			where $\Phi:(x,\yy,\zz)\mapsto(x,  h(x,\zz)\yy,x\zz)$ is a quasi-isomorphism as announced.
			
			If $Z$ and $\mathcal C(x)$ are tangent, we choose local coordinates 
			$(x,y_1,\yy,\zz)$
			such that $Z=\{y_1=0, \yy=0\}$,
			$c(x)=(c_1(x),c_{\yy}(x),c_{\zz}(x))$, where $c_1(x)=x^k(1+T_kc_1(x))$ has order $k>1$, 
			and $c_{\yy}$, $c_{\zz}$ have order (strictly) larger than $k$.
			In the usual  {coordinate} systems for the blow-up, 
			the lift of $\mathcal C$ by $\pi_Z$ has parameterization
			$(x,c_1(x),\frac{c_{\yy}(x)}{c_1(x)},c_{\zz}(x))$, 
			which can be expressed in terms of $T_kc(x)$:
			\[ \widetilde{\mathcal C}(x)=(x,x^k(1+T_kc_1(x)), \frac{T_kc_{\yy}(x)}{1+T_kc_1(x)}, x^kT_kc_{\zz}(x)).\]
			Since $(x,y_1,\yy,\zz)\mapsto (x, x^k(1+y_1),\frac{\yy}{1+y_1},x^k\zz)$
			is a quasi-isomorphism, we have the desired  {form} for $\widetilde{\mathcal C}(x)$.\end{proof}
		
		\begin{proof}[Proof of claim \ref{5:claim:2}]
			Since $T_{k+1}=T_{k}T_1$, we only treat the case $k=1$.
			We can suppose that $c(x)=xT_1c(x)$ has order at least $2$. 
			Writing $\psi(x,\yy)=\psi_0(x)+\sum_{|\alpha|\ge 1} \psi_{\alpha}(x)\yy^{\alpha}$, (the $\psi_i(x)$'s are vectors of same length as $c(x)$)
			we have 
			\[ T_1(\psi(x,c(x)))=T_1\psi_0(x)+\sum_{|\alpha|\ge 1} \psi_{\alpha}(x)x^{|\alpha|-1}(T_1c(x))^{\alpha},\]
			and we set $\phi(x,\yy)=T_1\psi_0(x)+\sum_{|\alpha|\ge 1} \psi_{\alpha}(x)x^{|\alpha|-1}\yy^{\alpha}$.
			Denote $\Phi(x,\yy)=(x,\phi(x,\yy))$. Then $\Phi$ is a quasi-isomorphism and $(x,T_1(\psi\circ c(x)))=\Psi(x,T_1c(x))$.\end{proof}
		
		\begin{proof}[Proof of claim \ref{5:claim:3}]  {According to} Claim \ref{5:claim:2}, we only need to prove it for $k=1$. We suppose with no loss of generality that $H(x)=xT_1H(x)$ has order greater than $1$, and write \[ T_1(H_i\circ P_j(x))= \frac{1}{x}H_i(P_j(x)) = \frac{H_i(P_j(x))}{P_j(x)}\frac{P_j(x)}{x}= \frac{P_j(x)}{x}(T_1H_i)\circ(P_j(x)).\] Since the polynomials $P_j$ have order at least $1$, 
			the mapping $\Psi$ is a  quasi-isomorphism with
			\[\Psi(x,\ww)=\left(x,\frac{P_j(x)}{x}w_{i,j},\; i=1\dots, n,\; j=1,\dots,l \right),\]
			and where $\ww=(w_{i,j},\; i=1\dots, n,\; j=1,\dots,l)$.\end{proof}
		It concludes the proof of Proposition \ref{pro:SAT-pencil}.
	\end{proof}

	\strut

	Proposition~\ref{pro:SAT-pencil} provides  {us with} examples of subanalytically transcendental curves in any dimension once we have a tuple of formal power series $H(x)$ with the SAT property. 
	In \cite[Theorem 2.4]{Rol-San-Sch},  {the authors} give conditions in order to guarantee that a formal solution of a system of analytic ordinary differential equations has the SAT property. These conditions only concern the eigenvalues of the linearization of the system and the Stokes phenomena of the formal solution. A particular case where their result applies is the well known Euler equation \[ 
	x^2y'=y-x,
	\]
	{whose}
	(unique) formal solution $E(x)=\sum_{n\ge0}n!x^{n+1}$ has the $q$-SAT property with $q=1$. This property and Proposition~\ref{pro:SAT-pencil}  {allows} to construct the following examples of non-interlaced integral pencils with subanalytically transcendental axis. Any trajectory of such pencils has the regular separation property by Proposition~\ref{pro:formal-axis}, so
	they illustrate the results of Section~\ref{sec:trajectories} and Corollary~\ref{cor:separated-trans-pencil}.
	
	\begin{example}\label{ex:pinceau1}
		{\em
			Consider the vector field
			\[ 
			\xi_1 =  2x^2\frac{\partial}{\partial x} +  2(y-x)\frac{\partial}{\partial y} +(z-2x)\frac{\partial}{\partial z}.
			\]
			The formal curve $\mathcal{C}(x)=(x,E(x),E(2x))$ is invariant by $\xi_1$ and subanalytically transcendental by Proposition~\ref{pro:SAT-pencil}. For any of the two associated half-curves $\mathcal{C}^\epsilon\subset\{\epsilon x>0\}$, $\epsilon\in\{+,-\}$, the integral pencil $\mc P_{\mathcal{C}^\epsilon}$ is separated by \cite[Pinceau final s\'epar\'e, section 4.4]{Can-Mou-San2}, since the eigenvalues $\{0,2,1\}$ of the linear part $d\xi_1(0)$ are all real and distinct. Note also that $x=0$ is the unstable manifold $W^u$ of $\xi_1$ and that $\CC(x)$ is the formal center manifold of $\xi_1$.
			
			The coordinate $x$ grows along any integral curve not contained in $W^u$.
			As a consequence, there exists a unique trajectory of $\xi$ contained in $\{x<0\}$ which accumulates to the origin, so that $\mc P_{\mathcal{C}^-}$ has a unique element (a pencil of ``dimension'' one). 
			{This is a solitary trajectory in the sense of \cite{LG:solitary}}.
			On the contrary, any trajectory issued  {from} a point in the half-space $\{x>0\}$  accumulates to the origin and its germ belongs to $\mc P_{\mathcal{C}^+}$ (a pencil of ``dimension'' three).\\
		}
	\end{example}
	\begin{example}\label{ex:pinceau2}
		{\em
			The vector field
			\[ 
			\xi_2 = x^2\frac{\partial}{\partial x} + (y-x)\frac{\partial}{\partial y} -(z+x)\frac{\partial}{\partial z}
		\]
			has as formal invariant curve $\mathcal{C}(x)=(x,E(x),E(-x))$.
			This curve is not  {included} in the situation of Proposition~\ref{pro:SAT-pencil} with $H(x)=E(x)$,
			since $P(x)=-x$ is not a positive 1-short polynomial. However, using \cite[Theorem 2.4]{Rol-San-Sch}, 
			we can observe that $H(x)=(E(x),E(-x))$ has the 1-SAT property: $H(x)$ is a formal 
			solution of a system of ODEs for which the linear part has eigenvalues $\{1,-1\}$ with distinct argument, 
			and  $H(x)$ has a nontrivial Stokes phenomenon along both singular directions 
			$\R_{>0}$ and $\R_{<0}$. We obtain by Proposition~\ref{pro:SAT-pencil} 
			that $\mathcal{C}$ is subanalytically transcendental. Since $\mbox{Spec}(d\xi_2(0))=\{0,1,-1\}$, 
			the vector field $\xi_2$ has unique center-stable $W^{cs}$ and center-unstable $W^{cu}$ manifolds, 
			both two-dimensional. Moreover, the germ of any trajectory of $\xi_2$ accumulating to $0$ must 
			be included either in $W^{cs}$ or in $W^{cu}$  (depending if  {it} accumulates to $0$ for 
			$t\to+\infty$ or $t\to-\infty$ for the time parameter $t$). Also, as in Example~\ref{ex:pinceau1}, 
			the coordinate $x$ grows on each trajectory not contained in the invariant surface $\{x=0\}$. 
			Consequently we obtain two separated pencils $\mc P_{\mathcal{C}^+}$, 
			$\mc P_{\mathcal{C}^-}$, associated to the corresponding half-branches 
			$\CC^\epsilon=\CC\cap\{\epsilon x > 0\}$ of $\CC$, both of ``dimension'' 2. 
			More precisely, $Y^-=W^{cs}\cap\{x<0\}$ (respectively $Y^+=W^{cu}\cap\{x>0\}$) 
			{\em realizes} the integral pencil $\mc P_{\CC^-}$ (respectively $\mc P_{\CC^+}$) 
			in the sense that any integral curve in $Y^\epsilon$ is a representative of an 
			element of $\mc P_{\CC^\epsilon}$ and each element of $\mc P_{\CC^\epsilon}$ 
			is contained in $Y^\epsilon$.
			It is worth to remark that neither $W^{cs}$  {nor} $W^{cu}$ is analytic, otherwise $\CC$ would not be analytically transcendental. 
		}
	\end{example}
	
	\begin{example}\label{ex:pinceau3}
		{\em
			Consider the analytic vector field 
			\[ 
			\xi_3 = x^2\frac{\partial}{\partial x} + (y-x)\frac{\partial}{\partial y}+\left(\frac{1+2x}{(1+x)^2} z -\frac{x(1+2x)}{1+x}\right)\frac{\partial}{\partial z}.
			\]
			It has the formal invariant curve $\mathcal{C}(x)=(x,E(x),E(x+x^2))$.
			The polynomial $P(x)=x+x^2$ is not a $1$-short polynomial. It is  {proven} in \cite[Lemme 3.1]{Rol-San-Sch} that there exists an analytic function $f(x,z_1,z_2)\not\equiv 0$ at the origin such that
			\[ 
			f(x,E(x),E(x+x^2))=0.
			\]
			Consequently, $\mathcal{C}\subset S_f:=\{f=0\}$ and it is not analytically transcendental. Note that $S_f$ has dimension two since $\mathcal{C}$ is not convergent. The integral pencil $\mc P_{\mathcal{C}^+}$, where $\mathcal{C}^+$ is the associated half-curve contained in $\{x>0\}$, is separated and contains the germ of any trajectory issued 
			{from} a point in $\{x>0\}$ in a small neighborhood of the origin. However, only those trajectories contained in $S_f$ have the property of regular separation (with respect to the polynomially bounded o-minimal structure $\RR=\R_{\rm an}$ of subanalytic sets).
		}
	\end{example}
	
	{Note} that the vector field $\xi_3$ in Example~\ref{ex:pinceau3} has rational coefficients,  {thus it is also definable} in the structure $\R_{\rm alg}$ of semi-algebraic sets. 
	{We suspect (but we did not prove it) that the surface $S_f$ is not semi-algebraic. If so, the formal curve would be transcendental with respect to semi-algebraic sets, and analogously to Proposition~\ref{pro:formal-axis},
		we could obtain the regular separation property for all trajectories with respect to semi-algebraic functions, 
		and  {apply} Theorem~\ref{th:main}  {for} the polynomially bounded o-minimal structure $\mathbb R_{\rm alg}$. The following example  {makes use of} this principle.}
	
	\begin{example}\label{ex:pinceau4}
		{\em
			Consider the polynomial vector field
			\[ 
			\xi = x^2\frac{\partial}{\partial x} + (y-x)\frac{\partial}{\partial y}+yz\frac{\partial}{\partial z}.
			\]
			For any $\mu\in\R\setminus\{0\}$, the formal curve $\mathcal{C}_\mu(x)=(x,E(x),\mu x\exp(E(x)))$ is invariant. Each $\mathcal{C}_\mu$ is not analytically transcendental since $\mathcal{C}_\mu$ is contained in the analytic surface $S_\mu=\{z-\mu y\exp(y)=0\}$. But $\mathcal{C}_\mu$ is algebraically transcendental: otherwise, if $\mathcal{C}_\mu$  {were} contained in some semialgebraic surface $B$ then $\mathcal{C}_\mu$  {would be} one of the components of the analytic curve $B\cap S_\mu$, which is impossible since $\mathcal{C}_\mu$ is divergent. 
		}
	\end{example}



\noindent\textbf{Funding.}\\
F. Sanz Sánchez was partially supported by Ministerio de Ciencia, Spain, process MTM2016-77642-C2-1-P  and PID2019-105621GB-I00.


\end{document}